\documentclass{article}
\usepackage{amsmath}
\usepackage{amsthm}
\usepackage{graphics}
\usepackage{epsfig}

\newtheorem{theorem}{Theorem}[section]
\newtheorem{cor}[theorem]{Corollary}
\newtheorem{lem}[theorem]{Lemma}
\newtheorem{prop}[theorem]{Proposition}

\theoremstyle{definition}
\newtheorem{defn}[theorem]{Definition}
\newtheorem{rem}[theorem]{Remark}

\newtheorem{example}[theorem]{Example}

\numberwithin{equation}{section}

\usepackage{amssymb}
\usepackage{amsfonts}
\usepackage{stmaryrd}
\usepackage{psfrag}
\usepackage{color}
\usepackage{url}
\usepackage{mathtools}

\newcommand{\N}{\mathbb{N}}
\newcommand{\R}{\mathbb{R}}
\newcommand{\T}{\mathbb{T}}
\newcommand{\Z}{\mathbb{Z}}
\newcommand{\Q}{\mathbb{Q}}
\newcommand{\C}{\mathbb{C}}
\newcommand{\B}{\mathbb{B}}
\newcommand{\mS}{\mathbb{S}}

\newcommand{\cB}{\mathcal{B}}
\newcommand{\cD}{\mathcal{D}}
\newcommand{\cE}{\mathcal{E}}

\newcommand{\cG}{\mathcal{G}}

\newcommand{\cI}{\mathcal{I}}
\newcommand{\cJ}{\mathcal{J}}
\newcommand{\cN}{\mathcal{N}}
\newcommand{\cP}{\mathcal{P}}

\newcommand{\cR}{\mathcal{R}}
\newcommand{\cT}{\mathcal{T}}
\newcommand{\cO}{\mathcal{O}}
\newcommand{\cZ}{\mathcal{Z}}
\newcommand{\cW}{\mathcal{W}}
\newcommand{\cU}{\mathcal{U}}

\begin{document}

\title{A characterization of Benford's Law\\in discrete-time linear
  systems}

\author{Arno Berger and Gideon Eshun\\[2mm] Mathematical and
    Statistical Sciences\\University of Alberta\\Edmonton, Alberta,
  {\sc Canada}}

\maketitle


%
\begin{abstract}
A necessary and sufficient condition (``nonresonance'') is
  established for every solution of an autonomous linear difference
  equation, or more generally for every sequence $(x^\top A^n y)$ with
  $x,y\in \R^d$ and $A\in \R^{d\times d}$, to be either trivial or
  else conform to a strong form of Benford's Law (logarithmic
  distribution of significands). This condition contains all pertinent
  results in the literature as special cases. Its number-theoretical
  implications are discussed in the context of
  specific examples, and so are its possible extensions and
  modifications.
\end{abstract}
\hspace*{8.6mm}{\small {\bf Keywords.} Benford sequence, uniform distribution
mod $1$, $\Q$-independence,}\\[-1mm]
\hspace*{27.5mm}{\small nonresonant set.}

\noindent
\hspace*{8.6mm}{\small {\bf MSC2010.} 37A05, 37A45, 11J71, 62E20.}

\medskip

\section{Introduction}

The study of digits generated by dynamical processes is a
classical subject that continues to attract interest
from many disciplines, including ergodic and number theory
\cite{ARS,CK,DK,KM, LagSou}, analysis \cite{BumEll, MasSch} and
statistics \cite{diek,GW,MiNi}. A recurring theme across the disciplines is the surprising ubiquity of a
logarithmic distribution of digits often referred to as {\em Benford's Law\/}
(BL). The most well-known special case of BL is the so-called ({\em decimal\/})
{\em first-digit law\/} which asserts that
\begin{equation}\label{eq1}
\mathbb{P} (\mbox{\em leading digit}_{10}\, = d_1) = \log_{10} \left( 1 +
  d_1^{-1} \right) \, , \quad \forall d_1 = 1 , \ldots , 9 \, ,
\end{equation}
where {\em leading digit}$_{10}$ refers to the leading (or first
significant) decimal digit, and $\log_{10}$ is
the base-$10$ logarithm (see Section \ref{sec2} for rigorous definitions); for example, the leading decimal
digit of $e =2.718$ is $2$, whereas the leading digit of $-e^{e}=-15.15$ is
$1$. Note that (\ref{eq1}) is heavily skewed towards the smaller
digits: For instance, the leading decimal digit is almost seven
times as likely to equal $1$ (probability $\log_{10} 2=30.10$\%) as it is
to equal $9$ (probability $1- \log_{10} 9 =4.57$\%).

Ever since first recorded by Newcomb \cite{newcomb} in 1881
and re-discovered by Benford \cite{benford} in 1938, examples of
data and systems conforming to (\ref{eq1}) in one form or another
have been discussed extensively, for instance in real-life data (e.g.\
\cite{doC, sam}), stochastic processes (e.g.\ \cite{schuerg}) and deterministic
sequences (e.g.\ $(n!)$ and the prime numbers \cite{Dia}). There now
exists a large body of literature devoted to the mechanisms whereby
mathematical objects, such as e.g.\ sequences or random variables, do
or do not satisfy (\ref{eq1}) or variants thereof. As of this writing, an online database \cite{BOB} devoted exclusively to BL lists more
than 800 references.

Due to their important role as elementary models throughout science,
{\em linear difference equations\/} have, from very early on, been studied for their conformance to (\ref{eq1}). A simple but
prominent case in point is the sequence $(x_n) =(1,1,2,3,5,\ldots )$ of Fibonacci
numbers, which has long been known \cite{BD, duncan, kuipers,wlod}  to
conform to (\ref{eq1}) in the sense that
\begin{equation}\label{eq2}
\lim\nolimits_{N\to \infty} \frac{\# \{ n \le N : \mbox{\em leading
    digit}_{10} ( x_n) = d_1 \}}{N} = \log_{10} (1+d_1^{-1}) \, , \quad \forall d_1 =
1,\ldots 9 \, .
\end{equation}
Recall that $(x_n)$ is a solution of a (very simple) {\em
  autonomous linear difference equation}, namely $x_n =
x_{n-1}+x_{n-2}$ for all $n\ge 3$.
This article provides a comprehensive theory of BL for such
equations. Specifically, the central question addressed (and answered) herein is this: Given $d\in \N$ and real numbers
$a_1, a_2, \ldots, a_{d-1}, a_d$ with $a_d\ne 0$, consider the (autonomous,
$d$-th order) linear difference equation
\begin{equation}\label{eq3}
x_n = a_1 x_{n-1} + a_2 x_{n-2}+ \ldots + a_{d-1} x_{n-d+1} + a_d
x_{n-d}  \, , \quad \forall n
\ge d+1 \, .
\end{equation} 
Under which conditions on $a_1, a_2, \ldots , a_{d-1}, a_d$, and presumably
also on the initial values $x_1,  \ldots , x_d$, does the solution
$(x_n)$ of (\ref{eq3}) satisfy (\ref{eq2})? There already exists a
sizeable literature addressing this question; see e.g.\ \cite{BDCDSA, KNRS,
  NSh, Sch88}. All previous work, however, seems to have led merely to {\em sufficient\/} conditions that are
either restrictive or difficult to state. By contrast, the main
result in this paper (Theorem \ref{thm380}) provides an
easy-to-state, {\em necessary and sufficient\/} condition for every
non-trivial solution of (\ref{eq3}) to satisfy (\ref{eq2}), and in
fact to conform to (\ref{eq1}) in an even stronger sense. The classical results
in the literature are then but simple corollaries.

To illustrate the main result, consider specifically the second-order
difference equation 
\begin{equation}\label{eqrec1}
x_n = 2 \gamma x_{n-1} - 5 x_{n-2} \, , \quad \forall n \ge 3 \, ,
\end{equation}
where $\gamma$ is a real parameter with $|\gamma|<\sqrt{5}$. Given any initial
values $x_1, x_2 \in \R$, does the solution $(x_n)$ of (\ref{eqrec1})
satisfy (\ref{eq2})? Theorem \ref{thm380} asserts that the answer to
this question is positive provided that the set $\cZ_{\gamma} = \{z^2
= 2\gamma z - 5\}= \{\gamma \pm \imath \sqrt{5 - \gamma^2}\}$ has a
certain number-theoretical property (``nonresonance''). For example,
if $\gamma = \sqrt{5} \cos (\pi /\sqrt{8})=0.9928$ then $\cZ_{\gamma}$
turns out to be nonresonant, and (\ref{eq2}) holds for every solution $(x_n)$ of
(\ref{eqrec1}), unless $x_1 = x_2 =0$, in which case $x_n \equiv 0$. On the other hand, if $\gamma =
\sqrt{5} \cos (\frac12 \pi \log_{10}5)=1.018$ then $\cZ_{\gamma}$
fails to be nonresonant, and correspondingly (\ref{eq2}) does not hold
for {\em any\/} solution of (\ref{eqrec1}). Finally, if $\gamma = 1$
then $(x_n)$ either satisfies (\ref{eq2}) for {\em all\/} initial
values $x_1,x_2$ (unless $x_1 = x_2 = 0$) or for {\em none\/} at all,
and experimental evidence seems to support the former alternative; see
Figure \ref{fig1} and also Example \ref{ex316b} below.

\begin{figure}[ht]
\begin{center}
\psfrag{td1}[]{$1$}
\psfrag{td2}[]{$2$}
\psfrag{td3}[]{$3$}
\psfrag{td4}[]{$4$}
\psfrag{td5}[]{$5$}
\psfrag{td6}[]{$6$}
\psfrag{td7}[]{$7$}
\psfrag{td8}[]{$8$}
\psfrag{td9}[]{$9$}
\psfrag{tna1}[r]{\small $29.99$}
\psfrag{tna2}[r]{\small $17.43$}
\psfrag{tna3}[r]{\small $12.82$}
\psfrag{tna4}[r]{\small $9.58$}
\psfrag{tna5}[r]{\small $7.93$}
\psfrag{tna6}[r]{\small $6.70$}
\psfrag{tna7}[r]{\small $6.01$}
\psfrag{tna8}[r]{\small $5.02$}
\psfrag{tna9}[r]{\small $4.52$}
\psfrag{tnb1}[r]{\small $43.68$}
\psfrag{tnb2}[r]{\small $7.77$}
\psfrag{tnb3}[r]{\small $7.20$}
\psfrag{tnb4}[r]{\small $6.87$}
\psfrag{tnb5}[r]{\small $6.68$}
\psfrag{tnb6}[r]{\small $6.62$}
\psfrag{tnb7}[r]{\small $6.71$}
\psfrag{tnb8}[r]{\small $6.99$}
\psfrag{tnb9}[r]{\small $7.48$}
\psfrag{tnc1}[r]{\small $29.99$}
\psfrag{tnc2}[r]{\small $17.23$}
\psfrag{tnc3}[r]{\small $12.78$}
\psfrag{tnc4}[r]{\small $9.51$}
\psfrag{tnc5}[r]{\small $7.92$}
\psfrag{tnc6}[r]{\small $6.61$}
\psfrag{tnc7}[r]{\small $6.01$}
\psfrag{tnc8}[r]{\small $5.19$}
\psfrag{tnc9}[r]{\small $4.76$}
\psfrag{tnd1}[r]{\small $30.10$}
\psfrag{tnd2}[r]{\small $17.60$}
\psfrag{tnd3}[r]{\small $12.49$}
\psfrag{tnd4}[r]{\small $9.69$}
\psfrag{tnd5}[r]{\small $7.91$}
\psfrag{tnd6}[r]{\small $6.69$}
\psfrag{tnd7}[r]{\small $5.79$}
\psfrag{tnd8}[r]{\small $5.11$}
\psfrag{tnd9}[r]{\small $4.57$}
\psfrag{taa1}[l]{$\gamma = 0.9928$}
\psfrag{tbb1}[l]{$\gamma = 1.018$}
\psfrag{tcc1}[l]{$\gamma = 1$}
\psfrag{tdd1}[l]{exact BL}
\includegraphics{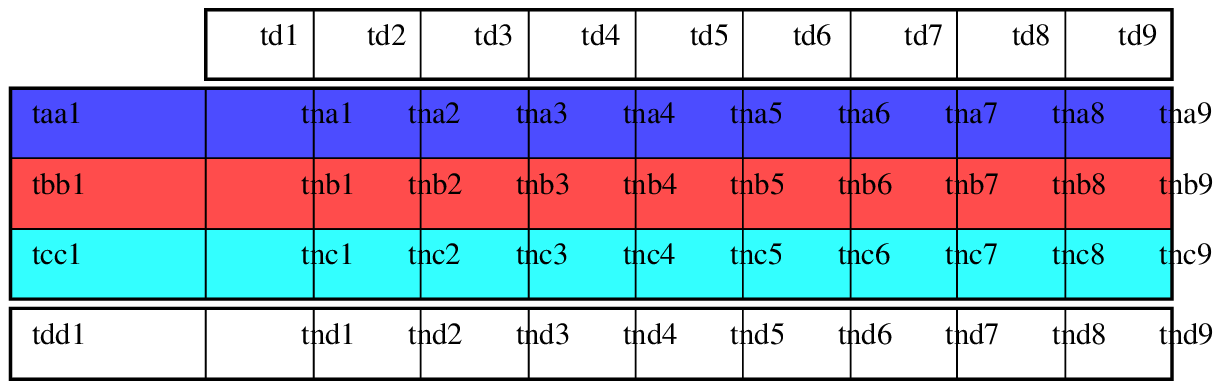}
%
%
\caption{Relative frequencies (in
    percent) of the leading decimal digits for the first $10000$ terms
  of the solution $(x_n)$ of (\ref{eqrec1}) with $x_1 = x_2 = 1$, for
  different values of the parameter $\gamma$; the
  bottom row shows the exact BL probabilities $100\cdot \log_{10}
  (1+d_1^{-1})$.}\label{fig1}
\end{center}
\end{figure}

This article is organized as follows. Section \ref{sec2} introduces
the formal definitions and analytic tools required for the
analysis. In Section \ref{sec3}, the main results are stated and
proved, based upon a tailor-made notion of nonresonance (Definition
\ref{def31}). Several examples are presented in order to illustrate this
notion as well as the main results. Finally, Section \ref{sec4}
briefly discusses possible extensions and modifications of the
latter. Given the widespread usage of discrete-time linear systems and
linear difference equations as models throughout the sciences, the
results of this article may contribute to a better understanding of,
and appreciation for BL and its applications in many disciplines. For
the reader's convenience, several analytical facts of an
auxiliary nature are deferred to an appendix, including the plausible
but lengthy-to-prove Theorem \ref{lemx} which in turn implies the crucial
Lemma \ref{lem2Omega}.

\section{Basic definitions and tools}\label{sec2}

Throughout this article, the following, mostly standard notation and
terminology is used. The symbols $\N$, $\N_0$, $\Z$, $\Q$,
$\R^+$, $\R$ and $\C$ denote the sets of, respectively, positive
integer, nonnegative integer, integer, rational, positive real, real
and complex numbers, and $\varnothing$ is the empty set. For every
integer $b\ge 2$, the logarithm base $b$ of $x\in \R^+$ is denoted
$\log_b x$, and $\ln x$ is the natural logarithm (base $e$) of $x$; for
convenience, let $\log_b0:=0$ for every $b$, and $\ln 0:= 0$. Given any $x\in \R$, the
largest integer not larger than $x$ is symbolized by $\lfloor x
\rfloor$. The real part, imaginary part, complex conjugate and
absolute value (modulus) of any $z\in \C$ is $\Re z$, $\Im z$,
$\overline{z}$ and $|z|$, respectively. For every $z\in \C\setminus
\{0\}$ there exists a unique number $-\pi < \arg z \le \pi$ with $z =
|z|e^{\imath \arg z}$. Given any
$w\in \C$ and $\cZ\subset \C$, define $w+\cZ:= \{w+z : z\in \cZ\}$ and $w\cZ:=
\{wz:z\in \cZ\}$. Thus with the unit circle $\mS:=\{ z\in \C : |z|=1\}$,
for example, $w + \mS = \{z\in \C: |z-w|=1 \}$ and $w\mS = \{z\in \C :|z| =|w|\}$ for every $w\in \C$.
The cardinality (number of
elements) of any finite set $\cZ\subset \C$ is $\#\cZ$.

The symbol $d$ throughout denotes a positive integer, usually unspecified or clear
from the context. The $d$-dimensional torus $\R^d /\Z^d$ is symbolized by
$\T^d$, its elements being represented as $\langle x \rangle = x+\Z^d$
with $x\in \R^d$; for simplicity write $\T$ instead of $\T^1$. The compact Abelian group $\T^d$ can be identified
with the $d$-fold product $\mS \times \ldots \times \mS$, via the
identification $\langle x \rangle = \langle (x_1, \ldots , x_d)\rangle \leftrightarrow (e^{2\pi \imath
  x_1}, \ldots , e^{2\pi \imath x_d})$ which is both a homeomorphism
(of compact spaces) and an isomorphism (of groups). Denote the Haar (probability) measure on $\T^d$ by
$\lambda_{\T^d}$. Call a set $\cJ\subset \T$ an {\em arc} if $\cJ=\langle
\cI \rangle := \{\langle x \rangle : x \in \cI\}$ for some interval
$\cI\subset \R$. With this, a sequence $(x_n)$ of real
numbers is {\em uniformly distributed modulo one}, henceforth
abbreviated as {\em u.d.} mod $1$, if
$$
\lim\nolimits_{N\to \infty} \frac{\# \{  n \le N : \langle x_n
  \rangle \in \cJ\}}{N} = \lambda_{\T}(\cJ) \quad \mbox{\rm for every
  arc $\cJ\subset \T$}\, .
$$
Equivalently, $\lim_{N\to \infty} \frac1{N}
\sum_{n=1}^N f(\langle x_n \rangle) = \int_{\T} f \, {\rm
  d}\lambda_{\T}$ holds for every continuous (or merely Riemann
integrable) function $f:\T \to \C$.

Recall that throughout $b$ is an integer with $b\ge 2$, informally
referred to as a {\em base}. Given a base $b$ and any $x\ne 0$, there exists a unique
real number $1\le S_b(x)<b$ and a unique integer $k$ such that $|x|=
S_b(x)b^k$. The number $S_b(x)$, referred to as the (base-$b$)
{\em significand} (or {\em mantissa}) of $x$, can be written
explicitly as
$$
S_b(x) = b^{\log_b|x| - \lfloor \log_b|x| \rfloor } \, ;
$$
in addition, let $S_b(0):=0$ for every base $b$. The integer $\lfloor
S_b(x) \rfloor$ is the {\em first significant digit\/} (base $b$) of
$x$; note that $\lfloor S_b(x)\rfloor \in \{1, \ldots , b-1\}$ whenever
$x\ne 0$. 

In this article, conformance to BL for sequences of real numbers is studied
via the following basic definition.

\begin{defn}
A sequence $(x_n)$ in $\R$ is a $b\,$-{\em Benford sequence}, or
$b\,$-{\em Benford\/} for short, with $b\in \N\setminus\{1\}$, if
$$
\lim\nolimits_{N\to \infty} \frac{\#\{ n \le N : S_b(x_n) \le
  s \}}{N} = \log_b s \, , \quad \forall s\in [1,b)  \, .
$$
The sequence $(x_n)$ is a {\em Benford sequence}, or simply {\em
  Benford}, if it is $b\,$-Benford for every $b\in \N \setminus \{1\}$.
\end{defn}

Specifically, note that (\ref{eq2}) holds whenever $(x_n)$ is
$10$-Benford, whereas the converse is not true in general since, for
instance, the sequence of first significant digits of $(2^n)$, i.e.\ $\bigl( \lfloor S_{10}(2^n)
\rfloor\bigr)=(2,4,8,1,3, \ldots)$, is
clearly not $10$-Benford yet can easily be shown to satisfy (\ref{eq2}).

Though very simple, the following observation is fundamental for the
purpose of this work because it enables the
application of a host of tools from the theory of uniform distribution.

\begin{prop}\label{prop_dia}{\rm $\!\!$\cite[Thm.1]{Dia}}
Let $b\in \N\setminus\{1\}$. A sequence $(x_n)$ in $\R$ is $b\,$-Benford if and
only if the sequence $(\log_b|x_n|)$ is u.d.\ {\rm mod} $1$.
\end{prop}

To prepare for the application of Proposition \ref{prop_dia}, several basic facts from the theory of
uniform distribution are reviewed here for the convenience of the
reader who, for an authoritative account on
the theory in general, may also wish to consult \cite{DT, KN}.

\begin{lem}\label{lem200}
The following are equivalent for every sequence $(x_n)$ in
$\R$:
\begin{enumerate}
\item $(x_n)$ is u.d.\ {\rm mod} $1$;
\item For every $\varepsilon > 0$ there exists a uniformly distributed
  sequence $(\widetilde{x}_n)$ with
$$
\overline{\lim}_{N\to \infty} \frac{\# \{ n \le N
  :|x_{n} - \widetilde{x}_{n}|> \varepsilon\}}{N} < \varepsilon \, ;
$$
\item Whenever $(y_n)$ converges in $\R$ then $(x_n + y_n)$ is u.d.\
  {\rm mod} $1$;
\item $(kx_n)$ is u.d.\ {\rm mod} $1$ for every $k\in \Z \setminus
  \{0\}$;
\item $(x_n + \alpha \ln n)$ is u.d.\ {\rm mod} $1$ for every $\alpha
  \in \R$.
\end{enumerate}
\end{lem}

\begin{proof}
Clearly (i)$\Rightarrow$(ii), and the converse is analogous to
\cite[Lem.2.3]{BDCDSA}.
Also, each of the statements (iii), (iv), and (v) trivially implies
(i), while the reverse implication is \cite[Thm.I.1.2]{KN},
\cite[Exc.I.2.4]{KN}, and \cite[Lem.2.8]{BDCDSA}, respectively.
\end{proof}

\begin{lem}\label{lem230}
Let $(x_n)$ be a sequence in $\R$, and $L\in \N$. If $(x_{nL+\ell })$ is
u.d.\ {\rm mod} $1$ for every $\ell \in \{1,\ldots , L\}$ then $(x_n)$ is
u.d.\ {\rm mod} $1$ as well.
\end{lem}

\begin{proof}
This follows directly from Weyl's criterion \cite[Thm.I.2.1]{KN}: For every $k\in \Z\setminus \{0\}$,
\begin{align*}
\left|
\frac1{N} \sum\nolimits_{n = 1}^N e^{2\pi \imath k x_{n}}
\right| & \le
\left|
\frac1{N} \sum\nolimits_{n = 1}^{L \lfloor N/L \rfloor} e^{2\pi \imath
  k x_{n}}
\right| +
\left|
\frac1{N} \sum\nolimits_{n = L \lfloor N/L \rfloor + 1}^N e^{2\pi \imath k x_{n}}
\right| \\[1mm]
& \le \left|
\frac1{N} \sum\nolimits_{\ell =1}^L \sum\nolimits_{n = 0}^{\lfloor N/L
  \rfloor - 1} e^{2\pi \imath k x_{nL + \ell}}
\right| + \frac{L}{N} \\[1mm]
& \le \frac1{L} \sum\nolimits_{\ell=1}^L \left|
\frac1{\lfloor N/L \rfloor} \sum\nolimits_{n=0}^{\lfloor N /L
  \rfloor - 1} e^{2\pi \imath  k x_{nL + \ell}}
\right| + \frac{L}{N} \: \stackrel {N\to \infty}{\longrightarrow} \: 0
\, ,
\end{align*}
because $\displaystyle \lim\nolimits_{M\to \infty}\frac1{M} \sum\nolimits_{n=0}^{M-1} e^{2\pi \imath k
  x_{nL + \ell}} = 0$ for every $\ell$, by assumption.
\end{proof}

When combined with the well-known fact that
$(n\vartheta)$ is u.d.\ mod $1$ precisely if $\vartheta \in \R$ is
irrational \cite[Exp.I.2.1]{KN}, Lemma \ref{lem200} and \ref{lem230}  immediately yield

\begin{lem}\label{lem240}
Let $\alpha, \vartheta \in \R$, $L\in \N$, and assume the sequence $(y_{n})$ in
$\R$ has the property that $(y_{nL+\ell})$ converges for every $\ell
\in \{1, \ldots , L\}$. Then $(n\vartheta + \alpha \ln
n + y_n)$ is u.d.\ {\rm mod} $1$ if and only if $\vartheta \in \R
\setminus \Q$.
\end{lem}

The remaining two results in this section deal with sequences of a
particular form that are going to appear naturally in later sections. 
For a concise formulation, given any $\cZ\subset \C$, 
denote by $\mbox{\rm span}_{\Q}\cZ$
the smallest subspace of $\C$ (over $\Q$) containing $\cZ$;
equivalently, if $\cZ\neq \varnothing$ then $\mbox{\rm span}_{\Q}\cZ$ is the set of all finite
{\em rational\/} linear combinations of elements of $\cZ$, i.e.
$$
\mbox{\rm span}_{\Q}\cZ = \bigl\{  \rho_1 z_1 + \ldots + \rho_n z_n
: n \in \N , \rho_1, \ldots , \rho_n \in \Q , z_1, \ldots , z_n
\in \cZ \bigr\} \, ;
$$
note that $\mbox{\rm span}_{\Q}\varnothing=\{0\}$. With this terminology,
recall that $z_1, \ldots, z_n\in \C$ are $\Q$-{\em independent}
(or {\em rationally independent}) if $\mbox{\rm span}_{\Q}\{z_1, 
\ldots, z_n\}$ is $n$-dimension\-al, or equivalently if $\sum_{j=1}^{n} k_j z_j= 0$ with
integers $k_1, \ldots , k_n$ implies $k_1 = \ldots = k_n=0$.
The following result is a generalization of
\cite[Lem.2.9]{BDCDSA}.

\begin{lem}\label{lem220}
Let $d\in \N$, $\vartheta_0, \vartheta_1, \ldots , \vartheta_d\in \R$,
and assume $f:\T^d \to \C$ is continuous, and non-zero
$\lambda_{\T^d}$-almost everywhere. If the $d+2$ numbers $1, \vartheta_0, \vartheta_1,
\ldots, \vartheta_d$ are $\Q$-independent then the sequence
$$
\Bigl( 
n\vartheta_0 + \alpha \ln n + \beta \ln \big|f\bigl(\langle ( n\vartheta_1, \ldots ,
n\vartheta_d) \rangle \bigr)  + z_n\big|
\Bigr)
$$
is u.d.\ {\rm mod} $1$ for every $\alpha, \beta \in \R$ and every
sequence $(z_n)$ in $\C$ with $\lim_{n\to \infty} z_n = 0$.
\end{lem}

\begin{proof}
For convenience, let
$$
x_n := n\vartheta_0 + \alpha \ln n + \beta \ln \big|f \bigl( \langle ( n\vartheta_1, \ldots ,
n\vartheta_d)\rangle \bigr)  + z_n\big| \, , \quad \forall n \in \N \, .
$$
The function $g:= \beta \ln |f|$
is continuous on a set of full $\lambda_{\T^{d}}$-measure, and so
\cite[Cor.2.6]{BDCDSA} together with Lemma \ref{lem200}(v) shows that the
sequence $(\widetilde{x}_n)$ with
$$
\widetilde{x}_n := n\vartheta_0 + \alpha \ln n + \beta \ln \big|f
\bigl( \langle (
n\vartheta_1 , \ldots , n\vartheta_d) \rangle\bigr) \big| \, , \quad \forall n \in \N \, ,
$$
is u.d.\ mod $1$ for every $\alpha, \beta \in \R$. Given $0<
\varepsilon \le 1$, choose $0< \delta < \frac12 \varepsilon/( 1 +|\beta|)$
so small that $\lambda_{\T^{d}}\bigl(  \{t\in \T^d : |f(t)| \le \delta
\}\bigr) < \varepsilon$. There exists $\cT\subset \T^d$ such that $\cT$ is a
finite union of open balls, 
$\cT \supset \{t\in \T^d :|f(t)|\le \delta \}$, and
$\lambda_{\T^{d}} (\cT)<\varepsilon$. Observe now
that if $\langle ( n\vartheta_1, \ldots , n \vartheta_d)\rangle \not \in
\cT$ and $|z_n|< \delta^2$ then
$$
|x_n - \widetilde{x}_n| =  |\beta| \left| \ln \bigg|  1 + 
    \frac{z_n}{f\bigl( \langle ( n\vartheta_1 , \ldots ,  n\vartheta_d)
      \rangle \bigr)} \bigg| \right|
\le 2 |\beta| \delta < \varepsilon \, .
$$
By the $\Q$-independence of $1, \vartheta_1, \ldots , \vartheta_d$,
the sequence $\bigl( ( n\vartheta_1 , \ldots,
n\vartheta_d) \bigr)$ is u.d.\ mod $1$ in $\R^d$, see e.g.\ \cite[Exp.I.6.1]{KN}, and so
$$
\lim\nolimits_{N\to \infty} \frac{\# \{ n \le N : \langle ( n
  \vartheta_1, \ldots , n \vartheta_d ) \rangle \in
  \cT \}}{N} =
\lambda_{\T^{d}}(\cT) < \varepsilon \, .
$$
With this and $\lim\nolimits_{n\to \infty} z_n = 0$, it follows that
\begin{align*}
& \overline{\lim}_{N\to \infty}  \frac{\# \{  n \le N : |x_{n}
  - \widetilde{x}_{n}| > \varepsilon\}}{N} \\[1mm]
& \le
\overline{\lim}_{N\to \infty}  \frac{\# \{ n \le N :  \langle ( n
  \vartheta_1 , \ldots , n \vartheta_d) \rangle  \in \cT
  \: \:  \mbox{\rm or} \: \: |z_{n}|\ge \delta^2 \}}{N} \\[1mm]
& \le
 \overline{\lim}_{N\to \infty}  \frac{\# \{ n \le N :  \langle (n
  \vartheta_1  , \ldots , n \vartheta_d) \rangle  \in \cT
  \}}{N}  + \overline{\lim}_{N\to \infty}  \frac{\# \{ n \le
  N :  |z_{n}|\ge \delta^2 \}}{N} \\[1mm]
& = \lambda_{\T^{d}}(\cT) + 0 < \varepsilon \, ,
\end{align*}
and an application of Lemma \ref{lem200}(ii) completes the proof.
\end{proof}

The assertion of the next, final lemma is very plausible indeed. Its
proof, however, is somewhat technical and hence deferred to an
appendix for the reader's convenience.

\begin{lem}\label{lem2Omega}
Let $d\in\N$, $p_1, \ldots , p_d\in \Z$, and $\beta \in \R \setminus
\{0\}$. Then there exists $u \in \R^d$
such that the sequence
$$
\Bigl( p_1 n \vartheta_1 + \ldots + p_d n \vartheta_d + \beta \ln
\big|u_1 \cos (2\pi n \vartheta_1) + \ldots + u_d \cos (2\pi n 
\vartheta_d) \big| \Bigr)
$$
is not u.d.\ {\rm mod} $1$ whenever $\vartheta_1, \ldots ,
\vartheta_d\in \R$ and the $d+1$ numbers $1, \vartheta_1, \ldots ,
\vartheta_d$ are $\Q$-independent.
\end{lem}

\begin{proof}
See Appendix A.
\end{proof}

\section{A Characterization of Benford's Law}\label{sec3}

Given a positive integer $d$ and real numbers $a_1, a_2, \ldots, a_{d-1}, a_d$ with
$a_d\ne 0$, consider the autonomous, $d$-th order linear difference equation (or recursion)
\begin{equation}\tag{\ref{eq3}}\label{eq3n1}
x_n = a_1 x_{n-1} + a_2 x_{n-2}+ \ldots + a_{d-1} x_{n-d+1} + a_d
x_{n-d}  \, , \quad \forall n
\ge d+1 \, . 
\end{equation} 
The goal of this section is to provide a necessary and sufficient
condition on the coefficients $a_1,
a_2, \ldots, a_{d-1}, a_d$ guaranteeing that every solution $(x_n)$ of (\ref{eq3n1})
is either Benford or trivial (identically zero); see Theorem
\ref{thm380} below. To make the analysis as transparent as possible, a
standard matrix-vector approach is utilized. Thus
associate with (\ref{eq3n1}) the matrix
\begin{equation}\label{eq3n2}
A= \left[
\begin{array}{ccccc}
a_1 & a_2 & \cdots & a_{d-1} & a_d \\
1 & 0 & \cdots & 0 & 0 \\
0 & 1 & 0 & \cdots & 0  \\
\vdots & \ddots & \ddots & \ddots & \vdots \\
0 & \cdots & 0 & 1 & 0
\end{array}
\right] \in \R^{d\times d} \, ,
\end{equation}
which is invertible since $a_d \ne 0$, and recall that, given
initial values $x_1, \ldots , x_d\in \R$, the solution $(x_n)$ of
(\ref{eq3n1}) can be expressed in the form
\begin{equation}\label{eq3n3}
x_n = (e^{(d)})^\top A^n y \, , \quad \mbox{where }\: y = A^{-1} \left[
\begin{array}{c}
x_d \\ \vdots \\ x_{1}
\end{array}
\right] \in \R^d \, ;
\end{equation}
here $e^{(1)}, \ldots , e^{(d)}$ represent the standard basis of $\R^d$;
$A^n$ is the $n$-th power of $A$, i.e.\ $A^n = A A^{n-1}$ for $n\ge 1$
and $A^0= I_d$, the $d\times d$-identity matrix; and $x^\top$ denotes the transpose of $x\in \R^d$, with $x^\top y$ being
understood as the real number $\sum_{j=1}^d x_j y_j$. As suggested by
(\ref{eq3n3}), in what follows, conditions are studied under which
$(x^\top A^n y)$ is $b$-Benford, where $x,y\in \R^d$ and $A$ is any given real $d\times
d$-matrix. Towards the end of this section, these conditions are, via
(\ref{eq3n2}) and (\ref{eq3n3}), specialized to solutions $(x_n)$ of
the linear difference equation (\ref{eq3n1}). Note that with $A$ given
by (\ref{eq3n2}), the sequence $(x^\top A^n y)$ is a solution of
(\ref{eq3n1}) for {\em every\/} $x,y\in \R^d$; see also the proof of
Lemma \ref{lem35} below.

As throughout the entire article, in the subsequent analysis of powers
of matrices, $d$ always denotes a fixed but otherwise unspecified positive integer.
For every $x\in \R^d$, the number $|x|\ge 0$ is the
{\em Euclidean norm\/} of $x$, i.e.\ $|x|=
\sqrt{x^\top x} = \sqrt{\sum_{j=1}^d x_j^2}$. A vector $x\in \R^d$ is
a {\em unit\/} vector if $|x|=1$. For every matrix $A\in\R^{d\times d}$, its spectrum, i.e.\ the set of its eigenvalues, is denoted by $\sigma(A)$.
Thus $\sigma(A)\subset \C$ is non-empty, contains at most $d$ numbers and is symmetric w.r.t.\ the real axis, i.e., all
non-real elements of $\sigma(A)$ come in complex-conjugate pairs. The number $r_{\sigma}(A):= \max\{|\lambda|:\lambda \in \sigma(A)\}\ge 0$
is the {\em spectral radius\/} of $A$. Note that $r_{\sigma}(A)>0$ unless $A$ is {\em nilpotent},
i.e.\ unless $A^N =0$ for some $N\in \N$; in the latter case $A^d = 0$
as well. For every $A\in \R^{d\times
  d}$, the number $|A|$ is the ({\em spectral\/})
{\em norm\/} of $A$ as induced by $|\cdot|$, i.e.\ $|A|
=\max \{|Ax| : |x|=1\}$. It is
well-known that $|A| = \sqrt{r_{\sigma} (A^\top A)}\ge r_{\sigma}(A) =
\lim_{n\to \infty} |A^n|^{1/n}$.

As will become clear shortly, some Benford properties related to
linear difference equations can be characterized in terms of the
spectrum of an associated matrix. The following
terminology turns out to be useful in this context.

\begin{defn}\label{def31}
Let $b\in \N\setminus \{1\}$. A non-empty set $\cZ \subset \C$ with
$|z|=r$ for some $r>0$ and all $z\in \cZ$, i.e.\ $\cZ\subset r\mS$, is
$b${\em -nonresonant\/} if the associated set 
\begin{equation}\label{eq31}
\Delta_{\cZ}:= \left\{
1 + \frac{\arg z - \arg w}{2\pi} : z,w \in \cZ
\right\} \subset \R 
\end{equation}
satisfies both of the following conditions:
\begin{enumerate}
\item $\Delta_{\cZ} \cap \Q = \left\{ 1 \right\}$;
\item $\log_b r \not \in \mbox{\rm span}_{\Q} \Delta_{\cZ}$.
\end{enumerate}
An arbitrary set $\cZ\subset \C$ is $b$-nonresonant if, for
every $r>0$, the set $\cZ\cap r\mS$ is either $b$-nonresonant or empty;
otherwise, $\cZ$ is $b${\em -resonant}.
\end{defn}

Note that the set $\Delta_{\cZ}$ in (\ref{eq31}) automatically satisfies
$1 \in \Delta_{\cZ} \subset (0,2)$ and is
symmetric w.r.t.\ the point $1$, i.e.\ $\Delta_{\cZ} = 2 -
\Delta_{\cZ}$. The empty set $\varnothing$ and the singleton $\{0\}$ are
$b$-nonresonant for every $b\in \N\setminus \{1\}$. Also, if $\cZ$ is
$b$-nonresonant then so is every $\cW \subset \cZ$. On the other hand,
$\cZ\subset \C$ is certainly $b$-resonant for every $b$ if either $\# (\cZ \cap r
\mS \cap \R) = 2$ for some $r>0$, in which case (i) is violated, or 
$\cZ\cap \mS\ne \varnothing$, which causes (ii) to fail. 

\begin{example}\label{exa32}
The singleton $\{z\}$ with $z\in \C$ is $b$-nonresonant if and only if either $z=0$
or $\log_b |z| \not \in \Q$. Similarly, any set $\{z, \overline{z}\}$
with $z\in \C \setminus \R$ is $b$-nonresonant if and only if $1$,
$\log_b |z|$ and $\frac1{2 \pi}\arg z$ are $\Q$-independent.
\end{example}

\begin{rem}\label{rem33}
(i) If $\cZ\subset
r\mS$ then, for every $z \in \cZ$,
$$
\mbox{\rm span}_{\Q} \Delta_{\cZ} = \mbox{\rm span}_{\Q} \left( 
\left\{ 1 \right\} \cup \left\{ 
\frac{\arg z - \arg w}{2\pi} : w \in \cZ
\right\}
\right) \, ,
$$
which shows that the dimension of $\mbox{\rm span}_{\Q} \Delta_{\cZ}$ as
a linear space over $\Q$ is at most $\# \cZ$. Also, if $\cZ\subset r\mS$ is symmetric w.r.t.\ the real axis, i.e.\
if $\overline{\cZ} = \cZ$, then condition (ii) in Definition
\ref{def31} is equivalent to $\log_b r \not \in \mbox{\rm span}_{\Q}
(\{1\} \cup \{\frac1{2 \pi} \arg z : z \in \cZ\})$; cf.\ \cite[Def.3.1]{BDCDSA}. 

(ii) The number $1$ in (\ref{eq31}) and part (i) of Definition
\ref{def31} has been chosen for convenience only; for the purpose of this
work, it could be replaced by any non-zero rational number.
\end{rem}

Recall that for the sequence $(xa^n y)$ with any $x,y\in \R$ and $a\in \R \setminus \{0\}$
to be either $b$-Benford (if $xy\ne 0$) or trivial (if
$xy=0$) it is necessary and sufficient that $\log_b|a|$ be
irrational. (This follows immediately e.g.\ from Proposition \ref{prop_dia} and
Lemma \ref{lem240}.) The following theorem, the first main result of
this article, extends this simple fact to arbitrary (finite) dimension by
characterizing the $b$-Benford property of $(x^\top A^n y)$ for any
$x,y\in \R^d$ and $A\in \R^{d\times d}$. To concisely formulate this and
subsequent results, call $(x^\top A^n y)$ and $(|A^nx|)$ with $x,y\in
\R^d$ and $A\in \R^{d\times d}$ {\em terminating\/} if, respectively, $x^\top A^n y = 0$
or $A^nx=0$ for all $n\ge d$; similarly, $(|A^n|)$ is terminating if
$A^n=0$ for all $n\ge d$. Also, recall that the asymptotic behaviour of $(A^n)$ is
completely determined by the eigenvalues of $A$, together with the
corresponding (generalized) eigenvectors. As far as Benford's Law base $b$ is
concerned, the key question turns out to be whether or not the set $\sigma
(A)$ is $b$-nonresonant. Notice that for $A=[a]\in \R^{1\times 1}$
with $a\ne 0$ the set $\sigma(A)=\{a\}$ is $b$-nonresonant if and only
if $\log_b|a|$ is irrational.

\begin{theorem}\label{thm34}
Let $A\in \R^{d\times d}$ and $b\in \N \setminus \{1\}$. Then
the following are equivalent:
\begin{enumerate}
\item For every $x,y\in \R^d$ the sequence $(x^\top A^n y)$
  is either $b$-Benford or terminating;
\item The set $\sigma(A)$ is $b$-nonresonant.
\end{enumerate}
\end{theorem}

The proof of Theorem \ref{thm34} is facilitated by two simple
observations, the first of which is an elementary fact from linear
algebra.

\begin{lem}\label{lem24a}
Let $L\in \{1, \ldots , d\}$ and assume $y^{(1)}, \ldots , y^{(L)}\in \R^d$ are linearly
independent. Then, given any $u \in \R^L$, there exists $x\in \R^d$
such that $x^\top y^{(\ell )} = u_{\ell}$ for every $1\le
\ell \le L$.
\end{lem}

\begin{proof}
The function
$$
\Phi : \left\{
\begin{array}{ccl}
\R^d & \to & \R^L \, , \\
x & \mapsto & \sum_{\ell=1}^L (x^\top y^{(\ell )}) e^{( \ell )} \, ,
\end{array}
\right.
$$
is linear, and since the (Gram) determinant
$$
\det [\Phi(y^{(1)}), \ldots
, \Phi (y^{(L)})] = \det[(y^{( \ell )})^\top y^{( k )}]_{\ell, k =
  1}^L
$$ is non-zero, $\Phi$ is also onto.
\end{proof}

A second observation clarifies the role of condition (i) in Definition \ref{def31} and
may also be of independent interest. Recall that a set $\cN \subset \N$
{\em has density\/} if 
$$
\rho(\cN) := \lim\nolimits_{N\to \infty} \frac{\#  \{ n \le N : n
  \in \cN \} }{N}
$$
exists. In this case, $\rho(\cN)$ is called the {\em density\/} of
$\cN$. Clearly, $0\le \rho(\cN) \le 1$ whenever $\cN$ has density. Not all
subsets of $\N$ have density, but those most relevant for Theorem
\ref{thm34} do.

\begin{lem}\label{lem35}
For every $A\in \R^{d\times d}$ and $x,y \in \R^d$, let
\begin{equation}\label{eq33}
\cN_{A,x,y}:= \{ n \in \N : x^\top A^n y = 0 \} \, .
\end{equation}
Then $\cN_{A,x,y}$ has density, and $\rho (\cN_{A,x,y}) \in \Q \cap [0,1]$.
\end{lem}

\begin{proof}
By the Cayley--Hamilton Theorem, there exist
$a_1, a_2, \ldots , a_{d-1}, a_d \in \R$ such that
$$
A^d = a_1 A^{d-1} + a_2 A_{d-2} + \ldots + a_{d-1} A + a_d I_d \, .
$$
Thus, for every $n\in \N$ and $x,y\in \R^d$,
\begin{align*}
x^\top  A^{n+d} y & = x^\top  (a_1 A^{n+d-1}  + a_2 A^{n+d-2} + \ldots + a_{d-1}A^{n+1} +
a_dA^n)y  \\
& = a_1 x^\top A^{n+d -1 } y + a_2 x^\top A^{n+d -2} y   + \ldots + a_{d-1} x^\top 
A^{n+1} y  + a_d x^\top A^n y  \, ,
\end{align*}
showing that $(x^\top A^n y)$ satisfies a linear $d$-step recursion relation
with constant coefficients. By the Skolem--Mahler--Lech Theorem
\cite[Thm.A]{MvdP}, the set $\cN_{A,x,y}$ is the union of a finite (possibly empty) set $\cN_0$
and a finite (possibly zero)
number of lattices, i.e.
\begin{equation}\label{eq34}
\cN_{A,x,y} = \cN_0 \cup \bigcup\nolimits_{\ell =1}^{L} \{ n N_{\ell} +
M_{\ell} : n\in \N \} \, ,
\end{equation}
where $L$ is a nonnegative integer, and $M_{\ell},N_{\ell} \in \N$
for $1\le \ell \le L$. From (\ref{eq34}) it is clear that $\cN_{A,x,y}$
has density, and $\rho (\cN_{A,x,y})$ is a rational number, in fact
$\rho(\cN_{A,x,y}) \cdot \mbox{\rm lcm} \{ N_1 , \ldots , N_{L} \} $
is a (nonnegative) integer.
\end{proof}

By using information about $\sigma(A)$, more can be said about the
possible values of $\rho (\cN_{A,x,y})$ in Lemma \ref{lem35}. In order
to concisely state the following observation, call a set $\cN \subset \N$ {\em co-finite\/} if $\N
\setminus \cN$ is finite. With this, $(x^\top A^n y)$ is
terminating precisely if $\cN_{A,x,y}$ is co-finite.

\begin{lem}\label{lem36}
For every $A\in \R^{d\times d}$ the following three statements are
equivalent:
\begin{enumerate}
\item For every $x,y\in \R^d$ the set $\cN_{A,x,y}$ in {\rm
    (\ref{eq33})} is either finite or co-finite;
\item $\rho (\cN_{A,x,y}) \in \{0,1\}$ for every $x,y \in \R^d$;
\item For every $r>0$ either $\Delta_{\sigma(A) \cap r\mS} \cap \Q =
  \{  1 \}$ or $\sigma (A) \cap r\mS = \varnothing$.
\end{enumerate}
\end{lem}

\begin{proof}
Clearly (i)$\Rightarrow$(ii), because $\rho(\cN)=0$ or $\rho(\cN)=1$
whenever $\cN$ is finite or co-finite, respectively.

Next, to establish the implication (ii)$\Rightarrow$(iii), assume (ii)
but suppose (iii) did not hold. (Note that this is possible only if $d\ge2$.) Thus $\#
(\Delta_{\sigma(A) \cap r\mS} \cap \Q)\ge 2$ for some $r>0$, which in
turn entails one of the following three possibilities: Either
\begin{equation}\label{eqss1}
\mbox{\rm both $-r$ and $r$ are eigenvalues of $A$,}
\end{equation}
or
\begin{equation}\label{eqss2}
\mbox{\rm $A$ has an eigenvalue $\lambda \in \C \setminus \R$ with
  $|\lambda| = r$ and $\frac1{2\pi} \arg \lambda>0$ rational,}
\end{equation}
or
\begin{align}\label{eqss3}
& \mbox{\rm $A$ has two eigenvalues $\lambda_1 , \lambda_2 \in \C \setminus \R$
  with $|\lambda_1| = |\lambda_2| = r$ and $\arg \lambda_1 > \arg \lambda_2 > 0$ }
\nonumber \\
& \mbox{\rm such
  that at least one of the two numbers $\frac1{2\pi} (\arg \lambda_1 \pm
  \arg \lambda_2)$ is rational.} 
\end{align}
Note that these cases are not mutually
exclusive, and (\ref{eqss3}) can occur only for $d\ge 4$. 

In case
(\ref{eqss1}), let $u,v\in \R^d$ be eigenvectors of $A$ corresponding
to the eigenvalues $-r,r$, respectively. Pick $x\in \R^d$ such that
$x^\top u = x^\top v = 1$. This is possible
because $u,v$ are linearly independent; see Lemma \ref{lem24a}. Then,
with $y:= u+v$,
$$
x^\top A^n y  = x^\top \bigl( (-r)^n u + r^n v \bigr) = r^n
\bigl( (-1)^n + 1 \bigr) \, , \quad \forall n \in \N \, ,
$$
showing that $\cN_{A,x,y} = \{2n-1 : n \in \N\}$. Thus $\rho (\cN_{A,x,y})
= \frac12\not \in \{0,1\}$, contradicting (ii). 

In case
(\ref{eqss2}), let $w\in \C^d$ be an
eigenvector of $A$ corresponding to the eigenvalue $\lambda$, and
observe that, for
every $n\in \N$,
\begin{align}\label{eqstar}
A^n \Re w  & = r^n \bigl(  \cos (n \arg \lambda) \Re w  -  \sin (n \arg
\lambda) \Im w \bigr) \, , \nonumber\\[-2.8mm]
& \\[-2.8mm]
A^n \Im w  & = r^n \bigl(  \sin (n \arg \lambda)  \Re w +  \cos (n \arg
\lambda) \Im w \bigr) \, . \nonumber
\end{align}
Again, since $\Re w , \Im w \in \R^d$ are linearly independent, it is possible to choose
$x\in \R^d$ such that $x^\top \Re w  = 1$ and $x^\top 
\Im w = 0$. With $y:= \Im w $, therefore,
$$
x^\top A^n y  = r^n \sin (n \arg \lambda) \, , \quad
\forall n \in \N \, .
$$
Since $\frac1{\pi} \arg \lambda$ is rational and strictly between
$0$ and $1$, the set $\cN_{A,x,y}$ equals $N\N$ for some integer
$N \ge 2$. Thus
$0< \rho (\cN_{A,x,y}) = \frac1{N} < 1$, again contradicting (ii). 

Lastly, in case (\ref{eqss3}) let $w^{(1)} ,w^{(2)}\in \C^d$ be eigenvectors of $A$
corresponding to the eigenvalues $\lambda_1, \lambda_2$, respectively.
As seen in (\ref{eqstar}) above, for
every $n\in \N$,
\begin{align*}
A^n\Re (w^{(1)}  + w^{(2)}) =  r^n \bigl( & \cos (n\arg \lambda_1) \Re w^{(1)}  -  \sin
(n\arg \lambda_1) \Im w^{(1)}   + \\
& + \cos (n\arg \lambda_2 ) \Re w^{(2)} -  \sin (n \arg \lambda_2
) \Im w^{(2)} \bigr) \, .
\end{align*}
Again, $\Re w^{(1)} , \Im w^{(1)}  , \Re w^{(2)} , \Im w^{(2)} \in \R^d$ are linearly
independent, and so by Lemma \ref{lem24a} it is possible to choose $x \in \R^d$ such that
$x^\top  \Re w^{(1)}   = -1$, $x^\top \Im w^{(1)}   = x^\top  \Im w^{(2)} = 0$, and
$x^\top \Re w^{(2)} = 1$. Then, with $y:= \Re ( w^{(1)} + w^{(2)} )$,
\begin{align*}
x^\top A^n y  & = r^n \bigl( \cos (n \arg \lambda_2)  - \cos (n
\arg \lambda_1) \bigr) \\
& = 2 r^n \sin \left( \pi n \frac{\arg \lambda_1
    - \arg \lambda_2}{2\pi} \right) \sin \left( \pi n \frac{\arg \lambda_1
    + \arg \lambda_2}{2\pi} \right) \, .
\end{align*}
Since both numbers $\frac1{2\pi} (\arg \lambda_1 \pm \arg
\lambda_2)$ are strictly between $0$ and $1$ and at least one of them is
rational, the set $\cN_{A,x,y}$ once more has a rational density that equals neither
$0$ nor $1$: From $\cN_{A,x,y} = N_1\N \cup N_2\N$ with two (not
necessarily different) integers $N_1,N_2\ge 2$, it follows that
$$
0 < \frac1{\min \{N_1,N_2\}} \le \rho (\cN_{A,x,y}) \le 1 - \frac1{\mbox{\rm
  lcm} \{N_1,N_2 \}} < 1
\, .
$$
Once again this contradicts (ii) and hence completes the
proof that indeed (ii)$\Rightarrow$(iii).

Finally, to show that (iii)$\Rightarrow$(i), denote the
``upper half'' of $\sigma(A)$ by
$$
\sigma^+(A) := \{ \lambda \in \sigma (A) : \Im \lambda \ge 0 \}
\setminus \{0 \} \, .
$$
Note that $\sigma^+(A)= \varnothing$ if and only if $A$ is nilpotent,
in which case clearly
$\cN_{A,x,y}$ is co-finite for all $x,y,\in \R^d$. From now on,
therefore, assume that $\sigma^+(A)\ne \varnothing$.
Recall that $A^n$ can be written in the form
\begin{equation}\label{prnew1}
A^n = \Re \left( \sum\nolimits_{\lambda \in \sigma^+(A)}
  P_{\lambda}(n) \lambda^n \right) \, , \quad \forall n \ge d \, ,
\end{equation}
where $P_{\lambda}$ is, for every $\lambda\in \sigma^+(A)$, a
(possibly non-real) matrix-valued polynomial of degree at most $d-1$,
i.e.\ $P_{\lambda}\in \C^{d\times d}$, and for all $j,k\in \{1,
\ldots, d\}$ the entry $[P_{\lambda}]_{jk} = (e^{(j)})^{\top} P_{\lambda}
e^{(k)}$ is a complex polynomial in $n$ of degree at most
$d-1$. Moreover, $P_{\lambda}$ is real, i.e.\ $P_{\lambda}\in
\R^{d\times d}$, whenever $\lambda \in \R$. The representation
(\ref{prnew1}) follows for instance from the Jordan Normal Form
Theorem. Deduce from (\ref{prnew1}) that
\begin{equation}\label{prrelast}
x^\top A^n y = \Re \left( \sum\nolimits_{\lambda \in \sigma^+(A)}
  x^\top P_{\lambda}(n)y \lambda^n  \right) =:  \Re \left( \sum\nolimits_{\lambda \in \sigma^+(A)}
  p_{\lambda}(n) \lambda^n  \right) \, , \quad \forall n \ge d \, ,
\end{equation}
with $p_{\lambda} = x^\top P_{\lambda} y$ being, for every $\lambda
\in \sigma^+(A)$, a (possibly non-real) polynomial in $n$ of degree at
most $d-1$. Clearly, if
$p_{\lambda}=0$ for every $\lambda\in \sigma^+(A)$ then $\cN_{A,x,y}$
is co-finite. From now on, therefore, assume that $p_{\lambda} \ne 0$
for at least one $\lambda\in \sigma^+ (A)$, i.e.
$$
r:= \max \{ |\lambda| : \lambda \in \sigma^+(A), p_{\lambda} \ne 0 \} > 0 \, .
$$
Denote by $k\in \N_0$ the maximal degree of the polynomials
$p_{\lambda}$ for which $|\lambda|=r$, i.e., let $k= \max\{ \mbox{\rm
  deg}\, p_{\lambda} : \lambda \in \sigma^+(A), |\lambda| = r \}$, and
consider the (non-empty) subset $\sigma^{++}$ of $\sigma^+(A)$ given by
$$
\sigma^{++} = \{\lambda \in \sigma^+(A) : |\lambda|=r, \mbox{\rm
  deg}\, p_{\lambda} = k \} \, .
$$
Note that $c_{\lambda}:= \lim_{n\to \infty}
p_{\lambda}(n)/n^k$ exists for every $\lambda \in \sigma^{++}$ and is
non-zero. With this, it follows from (\ref{prrelast}) that
\begin{equation}\label{eq3pfn1}
x^\top A^n y = r^n n^k \Re \left( 
\sum\nolimits_{\lambda \in \sigma^+(A)} \! \frac{p_{\lambda}(n)}{n^k} \left( \frac{\lambda}{r}\right)^n
\right)
=  r^n n^k \Re \left( 
\sum\nolimits_{\lambda \in \sigma^{++}} c_{\lambda} e^{\imath n \arg \lambda }
+ z_n
\right) \, ,
\end{equation}
where $(n z_n)$ is a bounded sequence in $\C$. Assume now that (iii)
holds but suppose $\cN_{A,x,y}$ was infinite. Then, by
(\ref{eq34}),
\begin{equation}\label{eq36}
\cN_{A,x,y} \supset \{nN+M : n\in \N \} \, , 
\end{equation}
with the appropriate $M,N\in \N$. Since $\lim_{n\to \infty} z_n
=0$, it follows from (\ref{eq3pfn1}) and (\ref{eq36}) that
\begin{equation}\label{eq3pfn2}
\lim\nolimits_{n\to \infty}\Re \left( 
\sum\nolimits_{\lambda \in \sigma^{++}} c_{\lambda} e^{\imath M \arg \lambda }
\left( e^{\imath N \arg \lambda}\right)^n
\right) = 0 \, .
\end{equation}
Since $c_{\lambda} e^{\imath M \arg \lambda }\ne 0$ for every $\lambda
\in \sigma^{++}$, Lemma \ref{lem34c} implies that either $e^{\imath N
  \arg \lambda_1} = e^{\pm \imath N \arg \lambda_2}$ for some
$\lambda_1, \lambda_2\in \sigma^{++}$ with $\lambda_1 \ne \lambda_2$,
or else $e^{\imath N \arg \lambda_1} = \pm 1$ for some $\lambda_1 \in
\sigma^{++}$. In the former case, at least one of the two numbers
$\frac{N}{2\pi}(\arg \lambda_1 \pm \arg \lambda_2)$ is a non-zero
integer, which in turn shows that $\# (\Delta_{\sigma(A) \cap r \mS}
\cap \Q)\ge 2$ and hence contradicts the assumed validity of (iii). In
the latter case, note first that $\# \sigma^{++}\ge 2$ because otherwise
either $c_{\lambda_1}=0$ (if $\sigma^{++} = \{\lambda_1\}\subset \R$),
which is impossible by the very definition of $\sigma^{++}$, or else
$\frac{N}{\pi} \arg\lambda_1$ is a non-zero integer (if $\sigma^{++} =
\{\lambda_1\}\subset \C \setminus \R$), which again contradicts
(iii). But $e^{\imath N \arg \lambda_1}=\pm 1$, together with
$\#\sigma^{++}\ge 2$ and (\ref{eq3pfn2}), leads to
$$
\lim\nolimits_{n\to \infty}\Re \left( 
\sum\nolimits_{\lambda \in \sigma^{++} \setminus \{\lambda_1\}}
c_{\lambda} e^{\imath M \arg \lambda }
\left( e^{ 2 \imath  N \arg \lambda}\right)^n
\right) = - \Re \left( c_{\lambda_1} e^{\imath M \arg
    \lambda_1}\right) \, ,
$$
and hence by Lemma \ref{lem34c} either $e^{2\imath N \arg \lambda_2}
=e^{\pm 2 \imath N \arg \lambda_3}$ for some $\lambda_2, \lambda_3 \in
\sigma^{++}\setminus \{\lambda_1\}$ with $\lambda_2\ne \lambda_3$, or else
$e^{2\imath N \arg\lambda_2} = \pm 1$ for some $\lambda_2 \in
\sigma^{++}\setminus \{\lambda_1\}$. As before, in the former case at least one
of the two numbers $\frac{N}{\pi} (\arg \lambda_2 \pm \arg \lambda_3)$
is a non-zero integer, contradicting (iii) again. Similarly, in
the latter case, $\frac{N}{\pi} \arg \lambda_1$ and
$\frac{2 N}{\pi}\arg \lambda_2$ are both integers, hence
$\frac1{2\pi}(\arg \lambda_1 - \arg \lambda_2)$ is rational and non-zero,
and this once more violates (iii). In summary, if (iii) holds then the set $\cN_{A,x,y}$ is necessarily
finite whenever $p_{\lambda}\ne 0$ for at least one $\lambda\in \sigma^+(A)$,
and, as seen earlier, it is co-finite otherwise. Thus (iii)$\Rightarrow$(i), and the proof is complete.
\end{proof}

\noindent
{\em Proof of Theorem \ref{thm34}:\/} To prove (i)$\Rightarrow$(ii),
assume $\sigma(A)$ is $b$-resonant. Then, for some $r>0$, either $\# (\Delta_{\sigma(A)
  \cap r \mS} \cap \Q)\ge 2$ or $\log_b r \in \mbox{\rm span}_{\Q}
\Delta_{\sigma(A) \cap r \mS}$, or both. In the former case, Lemma
\ref{lem36} guarantees the existence of $x,y\in \R^d$ for which
$0< \rho (\cN_{A,x,y})< 1$ and hence $(x^\top A^n y)$ is neither $b$-Benford nor terminating. 
As this clearly contradicts (i), it only remains to consider the case where $\# (\Delta_{\sigma(A)
  \cap r \mS} \cap \Q)\le 1$ for every $r>0$ yet $\log_b r_0 \in
\mbox{\rm span}_{\Q} \Delta_{\sigma (A) \cap r_0 \mS}$ for some
$r_0>0$. Label the elements
of $\sigma (A) \cap r_0 \mS$ as $\lambda_1, \ldots ,
\lambda_L$. Since $\overline{\sigma(A)} = \sigma(A)$,
$$
\log_b r_0 \in \mbox{\rm span}_{\Q} \Delta_{\sigma(A)\cap r_0 \mS}  = 
\mbox{\rm span}_{\Q} \left(  \{1\} \cup \left\{ \frac{\arg
      \lambda_{\ell}}{2\pi} : 1\le \ell \le L  \right\} \right) \, ;
$$
see Remark \ref{rem33}(i).
Let $L_0 + 1$ be the dimension (over $\Q$) of $\mbox{\rm
  span}_{\Q} \Delta_{\sigma (A) \cap r_0 \mS}$. Hence $L_0 \le L$, and
$L_0 \in \N$ unless $\frac1{2\pi} \arg \lambda_{\ell}$ is rational for
every $1\le \ell \le L$, in which case $L_0 = 0$. (For instance, the
latter inevitably occurs if $d=1$.) 

First consider the case of $L_0 = 0$. Here, $\log_b r_0$ and
$\frac1{2\pi}\arg \lambda_1$ are both rational, and
in fact $\lambda_1 \in \R$ because otherwise $\# (\Delta_{\sigma (A)
  \cap r_0 \mS} \cap \Q)\ge 2$. But then taking $x$ to be any eigenvector of $A$
corresponding to the eigenvalue $\lambda_1$ yields
$$
\log_b |x^\top A^n x| = \log_b (r_0^n |x|^2) = n \log_b r_0
+ 2 \log_b |x| \, ,
$$
which is periodic modulo one. Hence $(x^\top A^n x)$
is neither $b$-Benford nor terminating, a fact obviously contradicting
(i).

Assume from now on that $L_0 \ge 1$. In this case, by
re-labelling the eigenvalues $\lambda_1 , \ldots , \lambda_L$, it
can be assumed that $1, \frac1{2\pi} \arg \lambda_1 , \ldots ,
\frac1{2\pi} \arg \lambda_{L_0}$ are $\Q$-independent, and so
\begin{equation}\label{eq38}
\log_b r_0  = \frac{p_0}{q} + \frac{p_1}{q} \frac{\arg \lambda_1}{2\pi}
+ \ldots + \frac{p_{L_0}}{q} \frac{\arg \lambda_{L_0}}{2\pi} \, ,
\end{equation}
with the appropriate $p_0, p_1, \ldots, p_{L_0}\in \Z$ and $q\in
\N$. Let $w^{(1)} , \ldots ,
w^{(L_0)}\in \C^d$ be eigenvectors of $A$ corresponding to the
eigenvalues $\lambda_1, \ldots , \lambda_{L_0}$, respectively. 
Note that $\lambda_1 , \ldots, \lambda_{L_0}$
are all non-real, and consequently the $2L_0$ vectors $\Re w^{(1)} , \Im
w^{(1)} ,$ $ \ldots, $ $\Re w^{(L_0)}, \Im w^{(L_0)}$ are linearly
independent. Lemma \ref{lem24a} guarantees that, given any $u \in \R^{L_0}$, it is possible
to pick $x\in \R^d$ such that $x^\top \Re w^{(\ell )} = u_{\ell}$
and $x^\top  \Im w^{(\ell )} =0$ for all
$1 \le \ell \le L_0$. With $y:= \Re (w^{(1)}  + \ldots + w^{(L_0 )})$,
therefore, 
$$
x^\top A^n y  = r_0^n \bigl( 
u_1  \cos (n \arg \lambda_1) + \ldots + u_{L_0} \cos (n \arg \lambda_{L_0}) 
\bigr) \, , \quad \forall n \in \N \, .
$$
As $(x^\top A^n y)$ is not terminating whenever $u\ne 0 \in \R^{L_0}$,
Lemma \ref{lem36} shows that $x^\top A^n y \ne 0$ for all
sufficiently large $n$, and (\ref{eq38}) leads to 
\begin{align*}
q \log_b  |x^\top A^n y| = p_0 &  n  + p_1  n \frac{\arg \lambda_1}{2\pi}
+ \ldots + p_{L_0} n \frac{\arg \lambda_{L_0}}{2\pi}  + \\ & 
+
\frac{q}{\ln b} \ln \left| u_1 \cos \left( 2\pi n \frac{\arg
      \lambda_1}{2\pi} \right) + \ldots + u_{L_0}  \cos \left( 2\pi n \frac{\arg \lambda_{L_0}}{2\pi}\right)  \right| \, .
\end{align*}
Since $1, \frac1{2\pi} \arg \lambda_1, \ldots , \frac1{2\pi} \arg
\lambda_{L_0}$ are $\Q$-independent, by Lemma \ref{lem2Omega} one can specifically choose $u \in \R^{L_0}$ such that
$(q \log_b |x^\top A^n y| )$ is not u.d.\ mod 1, and hence
$( \log_b |x^\top A^n y|)$ is not u.d.\ mod 1 either, by
Lemma \ref{lem200}(iv). Thus $(x^\top A^n y)$ is neither Benford nor
terminating, a fact once again contradicting (i). Overall, therefore,
(i)$\Rightarrow$(ii), as claimed.

To prove the reverse implication (ii)$\Rightarrow$(i), let $\sigma(A)$ be
$b$-nonresonant. Given $x,y\in \R^d$, deduce from (\ref{prrelast})
that $(x^\top A^n y)$ is either terminating, or else
\begin{equation}\label{eq39}
x^\top A^n y  = |\lambda_1|^n n^k \Re \left( c_1 e^{\imath n
    \arg \lambda_1} + \ldots + c_L e^{\imath n \arg \lambda_L} + z_n
\right) \, , \quad \forall n \in \N  \, ,
\end{equation}
where $k\in \N_0$ and $L\in \N$; the numbers $\lambda_1 , \ldots ,
\lambda_L$ are appropriate (different) eigenvalues of $A$ with $|\lambda_1| =
\ldots = |\lambda_L| >0$ and $\Im \lambda_{\ell} \ge 0$ for all $1\le
\ell \le L$; the numbers $c_1, \ldots , c_L \in \C$ are all non-zero;
and $(n z_n)$ is a bounded sequence in $\C$. By the assumption of
$\sigma(A)$ being $b$-nonresonant,
$$
\log_b |\lambda_1| \not \in \mbox{\rm span}_{\Q} \Delta_{\sigma(A)
  \cap |\lambda_1| \mS} \supset \mbox{\rm span}_{\Q} \left(
\{1\} \cup \left\{ \frac{\arg \lambda_{\ell}}{2\pi} : 1 \le \ell \le L
\right\}
\right) \, .
$$ 
As before, let $L_0 + 1$ be the
dimension of $\mbox{\rm span}_{\Q} \left(\{1\} \cup
  \{\frac1{2\pi} \arg \lambda_{\ell} : 1\le \ell \le L \} \right)$, and
consider first the case of $L_0 = 0$, that is, $\frac1{2\pi} \arg
\lambda_{\ell}$ is rational for every $1\le \ell\le L$. As $\sigma
(A)$ would be $b$-resonant otherwise, this implies that $L=1$ and
$\lambda_1\in \R$. Since $\lambda_1$ is real, so is $c_1$, and for all
$n\in \N$, 
$$
|x^\top A^n y| = |\lambda_1|^n n^k \left|
\Re \left( 
c_1 e^{\imath n \arg \lambda_1} + z_n
\right)
\right| = |\lambda_1|^n n^k |c_1| \, \left| 1 + c_{1}^{-1} e^{-\imath
    n \arg \lambda_1} \Re z_n \right| \, .
$$
For all sufficiently large $n$, therefore,
$$
\log_b |x^\top A^n y | = n \log_b |\lambda_1| +
\frac{k}{\ln b} \ln n + \log_b|c_1| + \log_b \left| 1 + c_{1}^{-1} e^{-\imath
    n \arg \lambda_1} \Re z_n \right| 
$$
and since $\log_b |\lambda_1|$ is irrational, Lemmas \ref{lem200} and
\ref{lem240} imply that $(x^\top A^n y)$
is $b$-Benford.

It remains to consider the case of $L_0 \ge 1$. In this case, assume w.l.o.g.\ that $1,
\frac1{2\pi}
\arg \lambda_{1}, \ldots , \frac1{2\pi} \arg \lambda_{L_0}$ are
$\Q$-independent. Hence there exists $q\in \N$ and, for every $\ell \in
\{L_0+1 , \ldots , L\}$, an integer $p_{0{\ell}}$ as well as a vector
$p^{(\ell)}\in \Z^{L_0}$ such that
\begin{equation}\label{eq38a}
\frac{\arg \lambda_{\ell}}{2\pi} = \frac{p_{0{\ell}}}{q} + \frac{p_1^{(\ell)}}{q}
\frac{\arg \lambda_1}{2\pi} + \ldots + \frac{p_{L_0}^{(\ell)}}{q}
\frac{\arg \lambda_{L_0}}{2\pi} \, , \quad \forall  \ell \in \{ L_0 +
1 , \ldots , L \}  \, .
\end{equation}
Note that $p^{(\ell)}=0 \in\Z^{L_0}$ for at most one $\ell$, and
the $2L- L_0$ vectors 
$$
qe^{(1)}, \ldots , qe^{(L_0)}, \pm p^{(L_0 + 1)},
\ldots , \pm p^{(L)}\in \Z^{L_0}
$$ 
are all different because otherwise
$\sigma (A)$ would be $b$-resonant. As a consequence, for every $w\in \C^{L}$ the
multi-variate trigonometric polynomial $f_{w}:\T^{L_0} \to \R$ given by
$$
f_{w}(t)  = \Re \left(
\sum\nolimits_{\ell = 1}^{L_0} w_{\ell} e^{2\pi \imath q t_{\ell} } + \sum\nolimits_{\ell = L_0 +1}^L w_{\ell} e^{2\pi \imath
t^\top p^{(\ell)} } 
\right) 
$$
is non-constant, and so $f_{w}(t)\ne 0$ for
$\lambda_{\T^{L_0}}$-almost all $t\in \T^{L_0}$, provided that at least
one of the $L_0$ numbers $w_1, \ldots , w_{L_0}$ is non-zero. 

Fix now any $m \in \{1, \ldots , q\}$ and
deduce from (\ref{eq39}) and (\ref{eq38a}) that 
\begin{align*}
x^\top A^{nq+m}y  & = |\lambda_1|^{nq+m}  (nq+m)^k  \Re
\left( 
\sum\nolimits_{\ell = 1}^{L} c_{\ell} e^{\imath (nq + m ) \arg \lambda_{\ell}}
+ z_{nq+m}
\right) \\
& = |\lambda_1|^{nq} n^k |\lambda_1|^m \left( q + \frac{m}{n}
\right)^k  \Re
\left( 
\sum\nolimits_{\ell = 1}^{L_0} c_{\ell} e^{\imath m \arg \lambda_{\ell}}
e^{\imath nq \arg \lambda_{\ell}} \, +  \right. \\
& \qquad \qquad \qquad \quad 
\left. + \sum\nolimits_{\ell =L_0 + 1}^L c_{\ell} e^{\imath m \arg \lambda_{\ell}}
\prod\nolimits_{k = 1}^{L_0 }e^{\imath n p_{k}^{(\ell)} \arg
  \lambda_{k}}
+ z_{nq+m}
\right) \\
& =  |\lambda_1|^{nq} n^k |\lambda_1|^m \left( q + \frac{m}{n}
\right)^k \left( f_{w}\left( n \frac{\arg \lambda_1}{2\pi} ,
    \ldots ,  n \frac{\arg \lambda_{L_0}}{2\pi}\right)  + \Re
  z_{nq+m}\right) \, ,
\end{align*}
where $w \in \C^L$ is given by $w_{\ell} = c_{\ell} e^{\imath m \arg
  \lambda_{\ell}} \ne 0$ for all $\ell \in \{1, \ldots , L\}$.
Recall that by assumption the $L_0 + 2$ numbers $1, q \log_b |\lambda_1|, \frac1{2\pi}
\arg \lambda_1 , \ldots , \frac1{2\pi} \arg \lambda_{L_0}$ are
$\Q$-independent. Since $\lim_{n\to \infty} z_{nq+m} =0$ as well, Lemma
\ref{lem200} and \ref{lem220} applied to
\begin{align*}
\log_b |x^\top A^{nq+m} y| = nq \log_b|\lambda_1| & +
\frac{k}{\ln b} \ln n + m \log_b |\lambda_1| + k \log_b \left( q 
+  \frac{m}{n}\right) + \\
& + \frac1{\ln b} \ln \left| 
f_{w} \left( n \frac{\arg \lambda_1}{2\pi} ,
    \ldots ,  n \frac{\arg \lambda_{L_0}}{2\pi}\right)  + \Re z_{nq+m}
\right| 
\end{align*}
show that $(\log_b |x^\top A^{nq+m}  y  | )$ is u.d.\ mod
$1$. As $m\in \{1,\ldots , q\}$ was arbitrary, $(\log_b |x^\top A^{n} y|)$ is u.d.\ mod
$1$, by Lemma \ref{lem230}, i.e., $(x^\top A^n y)$ is
$b$-Benford. In summary, therefore, (ii)$\Rightarrow$(i), and the proof is complete.\hfill $\Box$

\begin{rem}\label{rem34a}
For {\em invertible\/} $A$ the important formula (\ref{prnew1}) holds
for all $n\in \N$. In this case, ``terminating'' 
in Theorem \ref{thm34}(i) can be replaced by ``identically zero''; see
also Corollary \ref{cor370} below.
\end{rem}

\begin{example}\label{ex39}
(i) The spectrum of $A=\left[ \begin{array}{cc} 1 & 1 \\ 1 & 0 \end{array} \right]$ is $\sigma (A) = \{\varphi , - \varphi^{-1}\}$ with
$\varphi = \frac12 (1+\sqrt{5})$. Since $A$ is invertible and $\log_b
\varphi$ is irrational (in fact, transcendental) for every $b\in \N
\setminus \{1\}$, the sequence $(x^\top A^n y)$ is, for every $x,y\in \R^2$, either Benford or identically zero. The latter
alternative occurs if and only if $x$ and $y$ are multiples of the (orthogonal) eigenvectors corresponding, respectively, to the eigenvalues 
$\varphi$ and $-\varphi^{-1}$, or vice versa.

(ii) Consider the (integer) $3\times 3$-matrix
$$
B = \left[
\begin{array}{rcc}
-3 & 1 & 0 \\
1 & 0 & 1 \\
0 & 1 & 6
\end{array}
\right] \, ,
$$
the characteristic polynomial of which is
$$
\chi_B (\lambda) = \det (B- \lambda I_3) = -\lambda^3 + 3\lambda^2 + 20 \lambda - 3 \, .
$$
Since $B$ is symmetric, all three eigenvalues of $B$ are real, and from $\chi_B(0)<0< \chi_B(1)$ it is clear
that they are all different. They also have different absolute values. To show that $\sigma (B)$ is $b$-nonresonant
for every $b\in \N \setminus \{1\}$, assume that $|\lambda|=b^{p/q}$
for some $\lambda \in \sigma(B)$ and relatively prime $p\in \Z\setminus \{0\}$, $q\in \N$.
If $p>0$ then $b^p$ is an eigenvalue of $B^q$ or $-B^q$ and hence divides $|\det B^q| =3^q$. This is only possible if $b=3^N$ for some
$N\in \N$. Similarly, if $p<0$ then $3^q b^{|p|}$ is an eigenvalue of
one of the two integer matrices $\pm (3B^{-1})^q$ and hence divides
$|\det (3B^{-1})^q|=3^{2q}$. Again, this leaves only the possibility of $b=3^N$ for some $N\in \N$. To analyse the
latter, assume now that $|\lambda|=3^{p/q}$ with relatively prime
$p\in \Z\setminus \{0\}$, $q\in \N$, possibly different from
before. Consider first the case of $p>0$. In this case, $\lambda$ is a root of one of the two irreducible polynomials
$\lambda^q \pm 3^p$ which in turn is a factor of $\chi_B$. Thus $q\le 3$, and since $3^{p}$ is an
eigenvalue of one of the two matrices $\pm B^q$, it follows that $p\le
q$. It can now be checked easily,
e.g.\ by computing $\chi_{B^2}$ and $\chi_{B^3}$, or by means of row reductions, that none
of the four numbers $\pm 3 ,\pm3^{2}$ is an eigenvalue of any of the
three matrices $B, B^2,B^3$. The possibility of $|\lambda|=3^{p/q}$
with $p<0$ is ruled out in a completely similar manner.
In summary, $\log_b |\lambda|$ is irrational for every $\lambda \in \sigma (B)$ and every $b\in \N \setminus \{1\}$, and $\sigma (B)$
is $b$-nonresonant. 
(Note that in order to draw this conclusion, it is not necessary to
explicitly know any eigenvalue of $B$.)
By Theorem \ref{thm34} and Remark \ref{rem34a}, the sequence $(x^\top B^n y)$ is, for every $x,y\in \R^3$, either
Benford or identically zero. As in (i), the latter case occurs
precisely if $x$ and $y$ (or vice versa) are, respectively, proportional and
orthogonal to the {\em same\/} eigenvector of $B$.

(iii) For the (invertible) matrix $C = \frac12 \left[
\begin{array}{cc}
1 + \pi  & 1 - \pi \\ 1- \pi & 1 + \pi
\end{array}
\right]$ the spectrum $\sigma(C)= \{1,\pi \}$ is $b$-resonant for every
$b\in \N \setminus \{1\}$. By Theorem \ref{thm34}, there exist $x,y\in
\R^2$ for which $(x^\top C^n y)$ is neither $b$-Benford nor identically
zero. Indeed, with $x=y= e^{(1)} + e^{(2)}$, for instance, $x^\top C^n y\equiv 2$. Similarly, $|C^n
x|\equiv \sqrt{2}$, so $(|C^nx|)$ as well is neither $b$-Benford nor
trivial. On the other hand, $(|C^n|)=(\pi ^n)$ is Benford. Theorems
\ref{thm350} and \ref{thm360} below relate these two simple observations to
the facts that $\sigma(C^n)=\{1,\pi ^n\}$ is $b$-resonant for every $n\in
\N$, whereas $\sigma(C) \cap r_{\sigma}(C)\mS = \{ \pi \}$ is not.
\end{example}

In addition to sequences of the form $(x^\top A^n y)$ in Theorem
\ref{thm34}(i), which may be thought of as {\em linear\/} observables of the
process $(A^n)$, some {\em non-linear\/} observables may also be of
interest. The next theorem establishes the Benford property
specifically for $(|A^n x|)$ with $x\in \R^d$. For the formulation of
the result, note that if $\cZ \subset \C$ is $b$-nonresonant then so is
$\cZ^n := \{z^n : z \in \cZ\}$ for every $n\in \N$. The converse does
not hold in general (unless $\# \cZ \le 1$), as the example of the $b$-resonant
set $\cZ= \{-\pi, \pi\}$ shows, for which $\cZ^2 = \{\pi^2\}$ is
$b$-nonresonant. Furthermore, this example illustrates the easily established
fact that $\cZ \subset r \mS$ satisfies (ii) of Definition
\ref{def31} if and only if $\cZ^N$ is $b$-nonresonant for some $N\in \N$.
Also, recall that $\sigma (A^n) =\sigma(A)^n$ for every
$A\in \R^{d\times d}$ and $n\in \N$.

\begin{theorem}\label{thm350}
Let $A\in \R^{d\times d}$ and $b\in \N \setminus \{1\}$. If $\sigma
(A^N)$ is $b$-nonresonant for some $N\in \N$ then, for every $x\in
\R^d$, the sequence $(|A^n x|)$ is either $b$-Benford or terminating.
\end{theorem}

\begin{proof}
Assume that $\sigma (A^N)$ is $b$-nonresonant and, as in the proof of
Lemma \ref{lem36}, consider the set
$\sigma^+(A) = \{\lambda \in \sigma (A): \Im \lambda \ge 0\}\setminus \{0\}$. If $\sigma^+ (A)=\varnothing$
then $A$ is nilpotent, $A^n x = 0$ for all $n\ge d$, and $(|A^n x|)$ is terminating. From now on,
therefore, assume that $\sigma^+ (A)\ne \varnothing$, and hence also $\sigma^+ (A^N)\ne \varnothing$.
Fix any $m\in \{1,\ldots , N\}$. Given $x\in \R^d$, deduce from
(\ref{prnew1}) with $A$ replaced by $A^N$ that
\begin{equation}\label{eq3t2n1}
A^{nN+m}x = 
\Re \left( 
\sum\nolimits_{\lambda\in \sigma^+(A^N)} P_{\lambda}(n)A^m x \lambda^n 
\right)
=:\Re \left(
\sum\nolimits_{\lambda \in \sigma^+(A^N)} q_{\lambda}(n) \lambda^n 
\right) \, , 
\end{equation}
for all $n\ge d$, where each $q_{\lambda}$ now is a (possibly non-real) vector-valued polynomial of degree at most
$d-1$, i.e., $q_{\lambda}(n)\in \C^{d}$, and every component of $q_{\lambda}$ is a polynomial in $n$ of
degree no larger than $d-1$. By the identical reasoning as in the
proof of the (iii)$\Rightarrow$(i) part in Lemma \ref{lem36}, deduce from (\ref{eq3t2n1}) that either
$A^{nN+m}x=0$ for all $n\ge d$, in which case $(|A^n x|)$ is terminating, or else, with the appropriate non-empty set $\sigma^{++}\subset \sigma^+(A^N)$ and
$c_{\lambda}\in \C^d \setminus \{0\}$ for every $\lambda \in \sigma^{++}$,
\begin{equation}\label{eq3t2n2}
A^{nN+m}x = r^n n^k \left(
\Re \left( \sum\nolimits_{\lambda\in \sigma^{++}} c_{\lambda} e^{\imath n \arg \lambda}\right) + u_n
\right) \, ,
\end{equation}
where $r>0$, $k\in \N_0$, and $(u_n)$ is a sequence in $\R^d$ for which $(n|u_n|)$ is bounded. (Note that $\sigma^{++}$,
and hence $r$, $k$, $c_{\lambda}$ and $(u_n)$ as well, may depend on $x$ and $m$.) Since $\sigma (A^N)$ is $b$-nonresonant,
$$
\log_b r \not \in \mbox{\rm span}_{\Q} \Delta_{\sigma (A^N) \cap r\mS} \supset
\mbox{\rm span}_{\Q} \left(
\{1\} \cup \left\{ \frac{\arg \lambda}{2\pi} : \lambda \in \sigma^{++}\right\} 
\right) \, .
$$
The argument now proceeds as in the proof of Theorem \ref{thm34}: Let $L_0+1$ be the dimension
of $\mbox{\rm span}_{\Q} \left(
\{1\} \cup \left\{ \frac1{2\pi}\arg \lambda  : \lambda \in \sigma^{++}\right\} 
\right)$. If $L_0=0$ then $\sigma^{++}=\{\lambda_1\}$ for some $\lambda_1 \in \R \setminus \{0\}$. (Otherwise
$\sigma(A^N)$ would be $b$-resonant.) In this case, $c_{\lambda_1}$ is
real as well, i.e.\ $c_{\lambda_1}\in \R^d$, and (\ref{eq3t2n2}) implies
$$
\log_b |A^{nN+m}x| = n \log_b r + \frac{k}{\ln b} \ln n + \log_b|c_{\lambda_1}+ e^{-\imath n \arg \lambda_1}u_n| \, ,
\quad \forall n \ge d \, .
$$
Since $\log_b r$ is irrational, $(\log_b |A^{nN+m}x|)$ is u.d.\ mod
$1$ by Lemma \ref{lem240}. The same argument can be
applied for every $m\in \{1,\ldots , N\}$, and so $(\log_b |A^n x|)$ is u.d.\ mod $1$ as well, by Lemma
\ref{lem230}. In other words, $(|A^n x|)$ is $b$-Benford.

Consider in turn the case of $L_0 \ge 1$. Label the elements of $\sigma^{++}$ as $\lambda_1, \ldots , \lambda_L$ and assume
w.l.o.g.\ that $1, \frac1{2\pi} \arg \lambda_1, \ldots , \frac1{2\pi}\arg \lambda_{L_0}$ are $\Q$-independent.
With the same notation as in (\ref{eq38a}), and given any vectors
$w^{(1)}, \ldots , w^{(L)} \in \C^d$, the vector-valued trigonometric polynomial $f:\T^{L_0}\to \R^d$ given by
$$
f_{w^{(1)}, \ldots , w^{(L)}}(t) = \Re\left(
\sum\nolimits_{\ell = 1}^{L_0} w^{(\ell )} e^{2\pi \imath q t_{\ell}} +
\sum\nolimits_{\ell = L_0+1}^L w^{( \ell )} e^{2\pi \imath t^\top p^{(\ell)}}
\right)
$$
is non-constant, provided that $w^{(\ell )}\ne 0$ for at least one
$\ell \in \{1, \ldots , L_0\}$. In this case, $f_{w^{(1)}, \ldots ,
  w^{(L)}} (t)\ne 0$, and hence also $|f_{w^{(1)}, \ldots ,
  w^{(L)}} (t)|\ne 0$ for $\lambda_{\T^{L_0}}$-almost all
$t\in \T^{L_0}$. Note that $|f_{w^{(1)}, \ldots ,
  w^{(L)}}|:\T^{L_0}\to \R$ is continuous. Fix now any $l \in \{1,
\ldots , q\}$, and deduce from (\ref{eq3t2n2}) that
\begin{align*}
& A^{(nq +l)N +m} x  = r^{nq} n^k r^l \left( q + \frac{l}{n}\right)^k
\biggl(
\Re \biggl(
\sum\nolimits_{\ell = 1}^{L_0} c_{\lambda_{\ell}} e^{\imath l \arg
  \lambda_{\ell}} e^{\imath n q \arg \lambda_{\ell}} +\\ 
& \qquad \qquad \qquad \quad 
 + \sum\nolimits_{\ell = L_0 + 1}^L c_{\lambda_{\ell}} e^{\imath l \arg
  \lambda_{\ell}} \prod\nolimits_{\nu = 1}^{L_0} e^{\imath n p^{(\ell)}_{\nu}
\arg \lambda_{\nu}}
\biggr)
+u_{nq+l}
\biggr)\\
& \qquad =  r^{nq} n^k r^l \left( q + \frac{l}{n}\right)^k \left(
f_{w^{(1)}, \ldots , w^{(L)}} \! \left( n \frac{\arg \lambda_1}{2\pi} ,
  \ldots , n \frac{\arg \lambda_{L_0}}{2\pi} \right) + u_{nq + l}
\right) \, ,
\end{align*}
with $w^{(\ell)} = c_{\lambda_{\ell}} e^{\imath l \arg
  \lambda_{\ell}}\in \C^d \setminus \{0\}$ for every $\ell \in \{ 1,
\ldots , L \}$. It follows that
$$
|A^{(nq +l)N +m}x| =  r^{nq} n^k r^l \left( q + \frac{l}{n}\right)^k
\Bigg|
\bigg| 
f_{w^{(1)}, \ldots , w^{(L)}} \! \left( n \frac{\arg \lambda_1}{2\pi} ,
  \ldots , n \frac{\arg \lambda_{L_0}}{2\pi} \right) 
\bigg| + z_n
\Bigg|\, ,
$$
where the (real) sequence $(z_n)$ is given by
\begin{align*}
z_n = & \left|  f_{w^{(1)}, \ldots , w^{(L)}} \! \left( n \frac{\arg \lambda_1}{2\pi} ,
  \ldots , n \frac{\arg \lambda_{L_0}}{2\pi} \right) + u_{nq + l}
\right| \\
& -
\left|  f_{w^{(1)}, \ldots , w^{(L)}} \! \left( n \frac{\arg \lambda_1}{2\pi} ,
  \ldots , n \frac{\arg \lambda_{L_0}}{2\pi} \right) 
\right| \, .
\end{align*}
Clearly, $|z_n|\le |u_{nq+l}|$, and so $\lim_{n\to
  \infty}z_{n}=0$. Lemmas \ref{lem200} and \ref{lem220}
now show that $(\log_b |A^{(nq + l)N + m}x|)$ is u.d.\ mod $1$. Since the number $l\in \{1, \ldots
q\}$ was arbitrary, $(\log_b |A^{nN+m}x|)$ is u.d.\ mod $1$ as well,
by Lemma \ref{lem230}. Moreover, the same argument can be applied for
every $m\in \{1, \ldots , N\}$, hence $(\log_b |A^n x|)$, too, is
u.d.\ mod $1$, i.e., $(|A^n x|)$ is $b$-Benford.
\end{proof}

In analogy to Theorem \ref{thm350}, the next result adresses the $b$-Benford property of the sequence
$(|A^n|)$. For a concise statement, the following terminology is
useful. Given any eigenvalue $\lambda$ of $A\in \R^{d\times d}$, let
$k(\lambda)\in \{0, \ldots , d-1\}$ be the largest integer for which
$$
\mbox{\rm rank} (A - \lambda I_d)^{k+1} < \mbox{\rm rank} (A - \lambda
I_d)^{k} \quad  \mbox{\rm if } \lambda \in \R \, ,
$$ 
and
$$
\mbox{\rm rank} (A^2 - 2\Re
\lambda A + |\lambda|^2 I_d)^{k+1} < \mbox{\rm rank} (A^2  - 2\Re \lambda
A + |\lambda|^2 I_d)^{k}
\quad  \mbox{\rm if } \lambda \in \C \setminus
\R\, . 
$$
Equivalently, $k(\lambda)+1$ is the size of the largest block
associated with the eigenvalue $\lambda$ in the Jordan Normal Form
(over $\C$) of $A$. With this, define the
{\em extremal peripheral spectrum\/} of $A$, henceforth denoted
$\sigma_{EP}(A)$, to be the set
\begin{equation}\label{eq359}
\sigma_{EP}(A) = \bigl\{
\lambda \in \sigma(A) \cap r_{\sigma}(A) \mS : k(\lambda) = k_{\max}
\bigr\} \, ,
\end{equation}
where $k_{\max} = k_{\max}(A) = \max \{ k(\lambda) : \lambda \in \sigma(A) \cap r_{\sigma}(A) \mS \}$.
Clearly $\sigma_{EP}(A) \subset \sigma(A)$, and just as
$\sigma(A)$, the set $\sigma_{EP}(A)$ is non-empty and symmetric
w.r.t.\ the real axis. Also, $\sigma_{EP}(A^n) = \sigma_{EP}(A)^n$ for
every $n\in \N$.

\begin{theorem}\label{thm360}
Let $A\in \R^{d\times d}$ and $b\in \N \setminus \{1\}$. If
$\sigma_{EP}(A^N)$ is $b$-nonresonant for some $N\in \N$ then either
$(|A^n|)$ is $b$-Benford or $A$ is nilpotent.
\end{theorem}

\begin{proof}
Clearly, $(|A^n|)$ is terminating if and only if $A$ is nilpotent.
Assume henceforth that $A$ is not nilpotent, thus $r_{\sigma}(A)>0$,
and let $\sigma_{EP}(A^N)$ be $b$-nonresonant. Fix any $m\in \{1,\ldots ,N\}$ and recall from
(\ref{prnew1}) that, in analogy to (\ref{eq3t2n1}) and (\ref{eq3t2n2}) above,
\begin{align}\label{eq3t3n1}
A^{nN+m}  & = \Re \left(
\sum\nolimits_{\lambda \in \sigma^+(A^N)} P_{\lambda}(n) A^m \lambda^n
\right)\nonumber \\
&  = r^n n^k 
 \left( \Re \left( 
\sum\nolimits_{\lambda \in \sigma^+(A^N) \cap r \mS} C_{\lambda} e^{i n \arg \lambda}
\right) + D_n \right)
\, , \quad \forall n\ge d \, ,
\end{align}
where $0<r \le r_{\sigma}(A^N) = r_{\sigma}(A)^N$ and $k\in \{ 0, \ldots
, d-1\}$ with $k\le k_{\max}(A) = k_{\max}(A^N)= :k_{\max}$, $C_{\lambda}\in
\C^{d\times d}$, and $(D_n)$ is a sequence in $\R^{d\times d}$ for
which $(n|D_n|)$ is bounded. (As in (\ref{eq3t2n2}) the quantities
$r$, $k$, $C_{\lambda}$ and $(D_n)$ may all depend on $m$.) From
(\ref{eq3t3n1}), it follows that
\begin{equation}\label{eq3t3n10}
|A^{nN+m}| \le r^n n^k a \, , \quad \forall n \in \N \, ,
\end{equation}
with the appropriate $a>0$. On the other hand, there exist $x,y\in \R^d$
for which
\begin{equation}\label{eq3t3n11}
\limsup\nolimits_{n\to \infty} \frac{|x^\top A^{nN + m}
  y|}{r_{\sigma}(A)^{nN+m} (nN + m)^{k_{\max}}} \ge 1 \, .
\end{equation}
Combining (\ref{eq3t3n10}) and (\ref{eq3t3n11}) yields
\begin{align*}
1 & \le \limsup\nolimits_{n\to \infty} \frac{|x^\top A^{nN + m}
  y|}{r_{\sigma}(A)^{nN+m} (nN + m)^{k_{\max}}}  \\
& \le
\frac{|x||y|a}{r_{\sigma}(A)^m N^{k_{\max}}} \limsup\nolimits_{n\to
\infty} \left( \frac{r}{r_{\sigma}(A)^N}\right)^n n ^{k - k_{\max}}
\, ,
\end{align*}
which in turn shows that $r = r_{\sigma}(A)^N$ and $k = k_{\max}$. With
$\sigma_{EP}^+(A^N) := \{\lambda \in \sigma_{EP}(A^N): \Im \lambda \ge
0\}$, therefore, (\ref{eq3t3n1}) can be re-written as
\begin{equation}\label{eq3t3n12}
A^{nN+m} = r_{\sigma}(A)^{nN} n^{k_{\max}} 
 \left( \Re \left( 
\sum\nolimits_{\lambda \in \sigma_{EP}^+(A^N)} C_{\lambda} e^{i n \arg \lambda}
\right) + E_n \right)
\, , \quad \forall n\in \N  \, ,
\end{equation}
where $C_{\lambda} \ne 0$ for some $\lambda \in
\sigma_{EP}^+(A^N)$, and $(n|E_n|)$ is bounded. Using (\ref{eq3t3n12}) and the $b$-nonresonance
of $\sigma_{EP}(A^N)$, completely analogous arguments as in the proof
of Theorem \ref{thm350} show
that $(\log_b|A^{nN+m}|)$ is u.d.\ mod $1$. Since $m\in \{1, \dots , N\}$ was arbitrary, $(\log_b |A^n|)$ is u.d.\ mod $1$
as well, i.e., $(|A^n|)$ is $b$-Benford.
\end{proof}

\begin{cor}\label{cor370}
Let $A\in \R^{d\times d}$ and $b\in \N \setminus \{1\}$. Assume that
$A$ is invertible and $\sigma(A)$ is $b$-nonresonant. Then:
\begin{enumerate}
\item For every $x,y\in \R^d$ the sequence $(x^\top A^n y)$ is
  $b$-Benford or identically zero;
\item For every $x\in \R^d \setminus \{0\}$ the sequence $(|A^n x|)$
  is $b$-Benford;
\item The sequence $(|A^n|)$ is $b$-Benford.
\end{enumerate}
\end{cor}

\begin{rem}
(i) Theorems \ref{thm350} and \ref{thm360}, and hence also Corollary
\ref{cor370}(ii,iii), hold similarly with $|\cdot |$ replaced by any norm on
$\R^d$ and $\R^{d\times d}$, respectively.

(ii) When comparing Theorems \ref{thm350} and \ref{thm360} to Theorem
\ref{thm34}, the reader may wonder what would happen to the latter if
in its statement (ii) $b$-nonresonance was assumed merely for $\sigma (A^N)$ with
some $N\ge 2$, rather than for $\sigma (A)$. The answer is simple:
With (ii) thus modified, (i)$\Rightarrow$(ii) of Theorem \ref{thm34}
would remain unchanged whereas the converse (ii)$\Rightarrow$(i) would
fail because unlike its analogues (\ref{eq3t2n2}) and
(\ref{eq3t3n12}), the representation (\ref{eq39}) with $A^n$ replaced
by $A^{nN+m}$ may no longer be valid. Note that this is in perfect agreement with the
fact, following from Lemma \ref{lem36}, that if $\sigma(A^N)$ is
nonresonant for some $N\ge 2$ yet $\sigma (A)$ is resonant then there
exist $x,y\in \R^d$ with $0< \rho (\cN_{A,x,y})<1$.
\end{rem}

The converses of Theorems \ref{thm350} and \ref{thm360} do not hold in
general: Even if $\sigma (A^n)$ and $\sigma_{EP}(A^n)$, respectively, are $b$-resonant for all $n\in \N$, the
sequence $(|A^nx|)$ nevertheless may, for every $x\in \R^d$, be
$b$-Benford or terminating, and $(|A^n|)$ may be $b$-Benford. In fact,
as the next example shows, it is impossible to characterize the
$b$-Benford property of $(|A^n x|)$ and $(|A^n|)$ solely
in terms of $\sigma (A)$ and $\sigma_{EP}(A)$, respectively ---
except, of course, for the trivial case of $d=1$.

\begin{example}\label{exa314}
For convenience, fix $b=10$ and consider the (invertible) $2\times 2$-matrix
$$
A = 10^{\pi}  \left[
\begin{array}{rr}
\cos (\pi^2) & - \sin (\pi^2)  \\
\sin (\pi^2) & \cos (\pi^2) 
\end{array}
\right] \, .
$$
The set $\sigma(A^n) =
\sigma_{EP} (A^n) =
\{10^{\pi n} e^{\pm \pi^2 \imath  n}\}$ is $b$-resonant for every $n\in
\N$ because
$$
\pi n  =   \log_{10}  10^{\pi n} \in \mbox{\rm
  span}_{\Q} \Delta_{\sigma(A^n)} = \mbox{\rm span}_{\Q} \{1,\pi \}
\, .
$$
Nevertheless, $10^{-\pi n} A^n$ is simply a rotation, hence $|A^n
x|=10^{\pi n}|x|$ for every $x\in \R^2$, and since $\log_{10}
10^{\pi} = \pi$ is irrational, $(|A^n x|)$ is $10$-Benford
whenever $x\ne 0$. Similarly, $(|A^n|)=(10^{\pi n})$ is
$10$-Benford. Thus the nonresonance assumptions in Theorems \ref{thm350} and \ref{thm360},
respectively, are not necessary for the conclusion.

Consider now also the (invertible) matrix
$$
B = \frac{10^{\pi }}{\sqrt{3}}  \left[
\begin{array}{rr}
\sqrt{3} \cos (\pi^2 ) & -  3 \sin ( \pi^2 )  \\[1mm]
  \sin ( \pi^2 )  & \sqrt{3} \cos (\pi^2 ) 
\end{array}
\right] \, ,
$$
for which $\sigma(B) = \sigma_{EP}(B) = \{10^{\pi} e^{\pm \pi^2
  \imath  }\} = \sigma (A)$, and so $\sigma (B^n) =\sigma_{EP}(B^n)= \sigma_{EP}(A^n)$
is $b$-resonant for every $n\in \N$. As far as spectral data are
concerned, therefore, the matrices $A$ and $B$ are
indistinguishable. (In fact, they are similar.) However, from
$$
B^n = \frac{10^{\pi n}}{\sqrt{3}}  \left[
\begin{array}{rr}
\sqrt{3 }\cos (\pi^2 n ) & - 3 \sin (\pi^2 n )  \\[1mm]
 \sin (\pi^2 n )   & \sqrt{3} \cos (\pi^2 n )   
\end{array}
\right] \, , \quad \forall n\in \N_0 \, ,
$$
it follows for instance that
$$
|B^n e^{(2)}| = 10^{\pi n} \sqrt{2 - \cos (2\pi^2 n)} \, , \quad
\forall n\in \N_0 \, ,
$$
and consequently
$$
\bigl\langle  \log_{10} |B^n e^{(2)}| \bigr\rangle = \left\langle
  \pi n   +{\textstyle \frac12} 
\log_{10} \bigl( 2 - \cos (2\pi^2 n  ) \bigr)\right\rangle = f( \langle n\pi
\rangle ) \, ,
$$
with the smooth function $f:\T \to \T$ given by
$$
f(t) = t + {\textstyle \frac12} \log_{10} \bigl( 2 - \cos (2\pi t)\bigr) \, .
$$
Recall that
$(n\pi)$ is u.d.\ mod $1$. Since $f$ is a diffeomorphism of $\T$ with
non-constant derivative, it follows that $\bigl(f(\langle n \pi
\rangle) \bigr)$ is not u.d.\ mod $1$, basically because $\lambda_{\T}
\circ f^{-1} \ne \lambda_{\T}$ (cf.\ Appendix A). Thus $(|B^n
e^{(2)}|)$, and in fact $(|B^n x|)$ for {\em every\/} $x\in
\R^2\setminus \{0\}$, is
neither $10$-Benford nor identically zero. Similarly,
$$
|B^n| = \frac{10^{\pi n}}{\sqrt{3}} \sqrt{ 4 - \cos (2\pi^2 n)
 + |\sin (\pi^2 n )|\sqrt{14 - 2
    \cos (2\pi^2 n )}} \, , \quad \forall n \in \N_0 \, ,
$$
and a completely analogous argument shows that $(|B^n|)$ is not
$10$-Benford either.
\end{example}

\begin{example}\label{exa314a}
Let again $b=10$ for convenience and consider the $6\times 6$-matrix
$$
A = \mbox{\rm diag} \left[
\left[
\begin{array}{cc} 2 & 1 \\ 0 & 2 \end{array}
\right]
,
\left[
\begin{array}{rr} -2 & 1 \\ 0 & -2 \end{array}
\right]
,
\frac{2}{\sqrt{3}}\left[
\begin{array}{rr} \sqrt{3} \cos (\pi \log_{10}2) & - 3 \sin (\pi
  \log_{10}2) \\ \sin (\pi \log_{10}2) & \sqrt{3} \cos (\pi \log_{10}2) \end{array}
\right]
\right] \! \, ,
$$
for which $\sigma(A) = \{\pm 2, 2 e^{\pm \pi \imath
  \log_{10}2}\}\subset 2 \mS$. Since
$$
\log_{10} 2^n = n \log_{10}2 \in \mbox{\rm span}_{\Q} \{ 1 ,
\log_{10}2 \} \subset \mbox{\rm span} _{\Q} \Delta_{\sigma(A^n)} \, , 
$$
the set $\sigma (A^n)$ is $b$-resonant for every $n\in
\N$. Correspondingly, there exist $x,y\in \R^6$ for which the
sequence $(x^\top A^n y)$, and in fact $(|A^n x|)$ as well, is neither $10$-Benford
nor terminating. Essentially the same calculation as in Example
\ref{exa314} shows that one can take for instance $x=y =
e^{(6)}$. Note, however, that $(x^\top A^n y)$ is $10$-Benford
whenever $|x_1y_2|\ne |x_3 y_4|$, hence for {\em most\/} $x,y\in
\R^6$; see also Theorem \ref{thm41} below.

On the other hand, since $k(\pm 2)=2$ and $k(2e^{\pm \pi \imath
  \log_{10}2})=1$, the set $\sigma_{EP}(A)$ equals $\{\pm 2\}$ which
is also $b$-resonant, yet $\sigma_{EP}(A^2)=\{4\}$ is $b$-nonresonant. By Theorem
\ref{thm360}, therefore, the sequence $(|A^n|)$ is $10$-Benford. This
could also have been demonstrated by means of Lemma \ref{lem240} and
an explicit calculation yielding
$$
|A^n| = 2^{n-1} n (1 +\alpha_n) \, , \quad \forall n \in \N \, ,
$$ 
where $(\alpha_n)$ is a sequence in $\R$ with $\lim_{n\to \infty} n^2
\alpha_n = 4$.
\end{example}

The final theorem in this section characterizes the $b$-Benford
property of solutions $(x_n)$ to linear difference equations
(\ref{eq3n1}). The result, which has informally been mentioned already
in the Introduction, follows directly from Theorem \ref{thm34}.

\begin{theorem}\label{thm380}
Let $a_1, a_2, \ldots , a_{d-1}, a_d$ be real numbers with $a_{d}\ne 0$, and $b\in \N \setminus \{1\}$. Then
the following are equivalent:
\begin{enumerate}
\item Every solution $(x_n)$ of {\rm (\ref{eq3n1})} is $b$-Benford
  unless $x_1 = x_2 = \ldots = x_d = 0$;
\item With the polynomial $p(z) = z^d - a_1 z^{d-1} - a_2 z^{d-2} - \ldots -
  a_{d-1}z - a_d$, the set $\{ z\in \C: p(z)=0\}$ is $b$-nonresonant.
\end{enumerate}
\end{theorem}

\begin{proof}
For convenience, let $\cZ:= \{z\in \C: p(z) = 0\}$. Note that $\cZ =
\sigma(A)$ for the matrix $A$ associated with (\ref{eq3n1}) via
(\ref{eq3n2}) because
\begin{align*}
\chi_A(z) & = \det (A - zI_d) = (-1)^d (z^d - a_1 z^{d-1} - a_2 z^{d-2} - \ldots -
  a_{d-1}z - a_d)\\
& = (-1)^d p(z) \, .
\end{align*}

To prove (i)$\Rightarrow$(ii), assume $\cZ$ is $b$-resonant. By
Theorem \ref{thm34} there exist $x,y\in \R^d$ for which $(x^\top A^n
y)$ is neither $b$-Benford nor terminating. Recall (e.g.\ from the
proof of Lemma \ref{lem35}) that $(x_n)$ with $x_n:= x^\top A^n y$
for all $n\in \N$ is a solution of (\ref{eq3n1}). By the choice of
$x,y$, the sequence $(x_n)$ is neither $b$-Benford nor terminating, let
alone identically zero. Hence (i) fails whenever (ii) fails, that is, (i)$\Rightarrow$(ii).

To establish the reverse implication (ii)$\Rightarrow$(i), recall from
(\ref{eq3n3}) that 
$$
x_n = (e^{(d)})^\top A^{n-1}y \, , \quad \forall n\in \N\, ,
$$
where $y = \sum_{j=1}^d x_{d+1 - j} e^{(j)}$. As $A$ is invertible, if $\cZ = \sigma (A)$ is $b$-nonresonant then, by Corollary \ref{cor370}, 
$(x_n)$ is either $b$-Benford or identically zero.
\end{proof}

\begin{example}\label{ex16a}
The set associated, via Theorem \ref{thm380}, with the familiar
difference equation
\begin{equation}\label{eq3ex316a}
x_n = x_{n-1} + x_{n-2} \, , \quad \forall n \ge 3 \, ,
\end{equation}
i.e.\ $\{ z\in \C : z^2 - z - 1 =0 \} = \{\varphi , - \varphi^{-1}\}$, is
$b$-nonresonant for every $b\in \N \setminus \{1\}$, see Example
\ref{ex39}(i). Except for the trivial solution $x_n \equiv 0$,
therefore, every solution $(x_n)$ of (\ref{eq3ex316a}) is Benford. This
contains as special cases the well-known sequences of Fibonacci and
Lucas numbers corresponding to the initial values $x_1
=x_2 =1$ and $x_1=2$, $x_2 = 1$, respectively.
\end{example}

\begin{example}\label{ex316b}
This example reviews, in the light of Theorem \ref{thm380}, the
second-order difference equation (\ref{eqrec1})
for the three specific values of the parameter $\gamma\in \R$ already
considered in the Introduction (recall Figure \ref{fig1}). For convenience, let
$b=10$ throughout. Note that the set associated with
(\ref{eqrec1}) is $\cZ = \cZ_{\gamma} = \{ z\in \C : z^2 = 2 \gamma z - 5\} =
\{\gamma \pm \imath \sqrt{5 - \gamma ^2}\}$, and so for
$|\gamma|<\sqrt{5}$ equals $\{\sqrt{5} e^{\pm \imath \arg z}\}\subset
\sqrt{5} \mS$ with
$\arg z = \arccos ( \gamma / \sqrt{5})  \in (0, \pi)$.

(i) Let $\gamma = \sqrt{5} \cos (\pi/ \sqrt{8} ) = 0.9928$. Then $\arg
z = \pi/\sqrt{8}$, and since
$$
  \log_{10} 5 \not \in \mbox{\rm span}_{\Q}
\Delta_{\cZ_{\gamma}} = \mbox{\rm span}_{\Q} \{1,\sqrt{2}\} \, ,
$$
the set $\cZ_{\gamma}$ is $b$-nonresonant. By Theorem \ref{thm380}, except for
$x_n \equiv 0$, {\em every\/} solution $(x_n)$ of (\ref{eqrec1}) is $10$-Benford.

(ii) Next, consider the case of $\gamma = \sqrt{5} \cos (\frac12 \pi
\log_{10}5)=1.018$. Now $\arg z = \frac12 \pi \log_{10}5$, and since
obviously
$$
 \log_{10}5 \in \mbox{\rm span}_{\Q}
\Delta_{\cZ_{\gamma}} = \mbox{\rm span}_{\Q} \{1,\log_{10} 5\} \, ,
$$
the set $\cZ_{\gamma}$ is $b$-resonant. It is clear
  that {\em no\/} solution of
(\ref{eqrec1}) is $10$-Benford in this case.

(iii) Finally, let $\gamma =1$. Here $\arg z = \arccos (1/\sqrt{5} ) =
\arctan 2$. It is not hard to see that $\frac1{\pi} \arctan 2$ is
irrational (as is, of course, $\log_{10} 5$). Thus, the
$b$-nonresonance of $\cZ_{\gamma}$ is equivalent to $\log_{10}5 \not \in
\mbox{\rm span}_{\Q} \{1, \frac1{\pi} \arctan 2\}$. It appears to be
unknown, however, whether the three numbers $1$, $\log_{10}5$, $\frac1{\pi}
\arctan 2$ are $\Q$-independent. If they are, then every non-trivial
solution of (\ref{eqrec1}) is $10$-Benford; otherwise none
is. As seen in Figure \ref{fig2}, numerical evidence seems to be in support
of the former alternative. (Rational independence of $1$, $\log_{10}
5$, and $\frac1{\pi} \arctan 2$, and thus $10$-nonresonance of
$\cZ_{\gamma}$ for $\gamma =1$ would follow immediately from Schanuel's
conjecture, a prominent but as yet unproven assertion in number theory
\cite[Sec.1.4]{wald}.)
\end{example}

\begin{figure}[ht]
\begin{center}
\psfrag{tp1}[]{\footnotesize $1$}
\psfrag{tpm0}[]{\footnotesize $0$}
\psfrag{tm1}[]{\footnotesize $-1$}
\psfrag{tp10}[]{\footnotesize $10$}
\psfrag{tp100}[]{\footnotesize $100$}
\psfrag{tp1000}[]{\footnotesize $1000$}
\psfrag{tp10000}[]{\footnotesize $10000$}
\psfrag{tt1}[l]{$x_n = 2\gamma  x_{n-1} - 5 x_{n-2}\, ,\quad \forall n\ge 3$}
\psfrag{tt2}[l]{$x_1 = x_2 = 1$}
\psfrag{tanb}[l]{\textcolor{black}{$\gamma =1.018$}}
\psfrag{tanb2}[l]{\textcolor{red}{{\bf not Benford}}}
\psfrag{tabf}[l]{\textcolor{black}{$\gamma = 0.9928 $}}
\psfrag{tabf2}[l]{\textcolor{blue}{{\bf Benford}}}
\psfrag{taun}[l]{\textcolor{black}{$\gamma = 1$}}
\psfrag{taun2}[l]{\textcolor{cyan}{{\bf unknown}}}
\psfrag{tlogdel}[l]{$\log \beta_N$}
\psfrag{tnn}[]{$N$}
\psfrag{betan}[l]{\small $\beta_N = \max_{d_1=1}^9 \left| {\displaystyle \frac{
    \# \{ n \le
      N : \mbox{\em leading digit}_{10} (x_n)  = d_1\}}{N}}  - \log_{10}(1 +d_1^{-1})\right|$}
%
%
\includegraphics{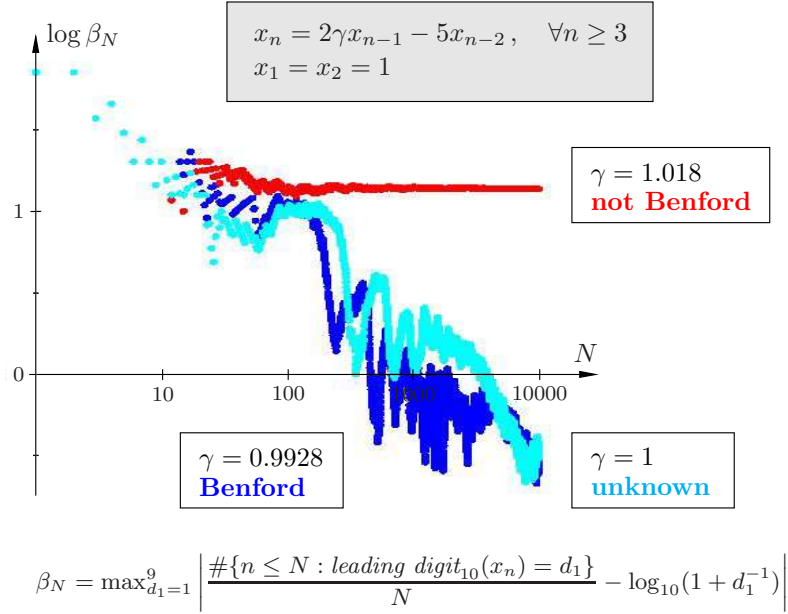}
\caption{For different values of the parameter $\gamma$, the solutions $(x_n)$ of
  (\ref{eqrec1}) may or may not be $10$-Benford; see Example
  \ref{ex316b} and also Figure \ref{fig1}.}\label{fig2}
\end{center}
\end{figure} 

\begin{rem}
Earlier, weaker forms and variants of the implication (ii)$\Rightarrow$(i) in
Theorems \ref{thm34} and \ref{thm380}, or special cases thereof, can
be traced back at least to \cite{NSh} and may also be found in
\cite{BDCDSA, BHPS, BS, KNRS, Sch88}. The reverse implication
(i)$\Rightarrow$(ii) seems to have been addressed previously only for $d<4$; see
\cite[Thm.5.37]{BHPS}. For the special case of $b=10$, partial proofs of Theorems
\ref{thm34} and \ref{thm380} have been presented in \cite{BE, BHPUP}.
\end{rem}

\section{Further examples and concluding remarks}\label{sec4}

This final section illustrates how key results of this
article (Theorems \ref{thm34} and \ref{thm380}) may take a
significantly different (and arguably simpler) form if
either their conclusion is weakened slightly or one additional assumption
is imposed. Concretely, Theorem \ref{thm34} for instance may be weakened in that
its $b$-Benford-or-terminating dichotomy (i) is assumed to hold only for
(Lebesgue) {\em almost all\/} $(x,y)\in \R^{d}\times
\R^{d}$. Alternatively, it may be assumed that the
matrix $A^N$ is {\em positive\/} for some $N\in \N$. As detailed
below, either of these modifications gives rise to
new forms of the results that may be of independent interest.

As throughout, $b\ge 2$ is a positive integer, and given any $A\in
\R^{d\times d}$, let
$$
\B_b (A):= \bigl\{(x,y)\in \R^{d}\times \R^{d} : (x^\top A^n y) \, \mbox{\rm is
$b$-Benford}\bigr\}\, .
$$
Denote Lebesgue measure on $\R^d \times \R^d$ by $\mbox{\rm
  Leb}_{d,d}$. Also recall from (\ref{eq359}) the definition of the
extremal peripheral spectrum $\sigma_{EP}(A)$. Although
$\sigma_{EP}(A)$ may constitute only a small part of $\sigma(A)$, it
nevertheless controls the Benford property of {\em most\/} sequence
$(x^\top A^n y)$. More precisely, the following variant of Theorem \ref{thm34}
holds.

\begin{theorem}\label{thm41}
Let $A\in \R^{d\times d}$ and $b\in \N\setminus \{1\}$. Assume $A$ is not nilpotent. Then the following are equivalent:
\begin{enumerate}
\item For almost every $(x,y)\in \R^d \times \R^d$ the sequence
  $(x^\top A^n y)$ is $b$-Benford, i.e., $\R^d \times \R^d \setminus
  \B_b(A)$ is a $\mbox{\rm Leb}_{d,d}$-nullset;
\item The set $\sigma_{EP}(A^N)$ is $b$-nonresonant for some $N\in \N$.
\end{enumerate}
\end{theorem}

\begin{proof}
To demonstrate (i)$\Rightarrow$(ii), assume that
$\sigma_{EP}(A^n)=\sigma_{EP}(A)^n$ is $b$-resonant for every $n\in
\N$, and hence $\log_b r_{\sigma}(A)\in \mbox{\rm span}_{\Q}
\Delta_{\sigma_{EP}(A)}$. In analogy to (\ref{eq3t3n12}), write
\begin{equation}\label{eq4t1p1}
A^n = r_{\sigma}(A)^n n^{k_{\max}} \left(
\Re \left( 
\sum\nolimits_{\lambda \in \sigma_{EP}^+(A)} C_{\lambda} e^{\imath n
  \arg \lambda} 
\right) + E_n
\right) \, , \quad \forall n \in \N \, ,
\end{equation}
where $C_{\lambda} \in \C^{d\times d}$ for every $\lambda \in
\sigma_{EP}^+(A)$, and $(E_n)$ is a sequence in $\R^{d\times d}$ for
which $(n|E_n|)$ is bounded. If $C_{\lambda}=0$ for
all $\lambda\in \sigma_{EP}^+(A)$, then (\ref{eq4t1p1}) would imply that
$$
\lim\nolimits_{n\to \infty} \frac{|A^n|}{r_{\sigma}(A)^n
  n^{k_{\max}}} = 0 \, ,
$$
whereas on the other hand there always exist $x,y\in \R^d$ with
$$
1 \le \limsup\nolimits_{n\to \infty} \frac{|x^\top A^n y|}{r_{\sigma}(A)^n
  n^{k_{\max}}} \le |x| \, | y | \limsup\nolimits_{n\to \infty} \frac{|A^n|}{r_{\sigma}(A)^n
  n^{k_{\max}}} \, .
$$
This contradiction shows that $C_{\lambda}\ne 0$ for some $\lambda\in
\sigma_{EP}^+(A)$.

Similarly to the proofs in the previous section, let $L_0 + 1$ be the
dimension of $\mbox{\rm span}_{\Q}\Delta_{\sigma_{EP}(A)}$ and consider
first the case of $L_0 = 0$. Here, with the appropriate $q\in \N$, the
numbers $q\log_b r_{\sigma}(A)$ and $q\frac1{2\pi} \arg \lambda$ for
all $\lambda \in \sigma_{EP}^+(A)$ are integers, and so
(\ref{eq4t1p1}) takes the form
\begin{equation}\label{eq4t1p2}
A^n = r_{\sigma}(A)^n n^{k_{\max}} (B_n + E_n) \, , \quad \forall n
\in \N \, ,
\end{equation}
where the sequence $(B_n)$ in $\R^{d\times d}$ is $q$-periodic, i.e.\
$B_{n+q} = B_n$ for all $n\in \N$. Suppose that $B_{\ell} = 0$ for
some $\ell \in \{1, \ldots , q\}$. Then
$$
\lim\nolimits_{n\to \infty} \frac{|A^{nq + \ell}|}{r_{\sigma}(A)^{nq +
  \ell} (nq + \ell)^{k_{\max}}} = 0 \, ,
$$
whereas similarly as before,
$$
\limsup\nolimits_{n\to \infty} \frac{|x^\top A^{nq + \ell}
  y|}{r_{\sigma}(A)^{nq + \ell} (nq + \ell)^{k_{\max}}} \ge 1 
$$
with the appropriate $x,y\in \R^d$. This contradiction shows that $B_{\ell}\ne 0$
for every $\ell \in \{1, \ldots , q\}$. Consequently, for each $\ell$ the set
$$
\cR_{\ell} := \bigl\{ (x,y) \in \R^d \times \R^d : x^\top B_{\ell} \,
y = 0 \bigr\}
$$
is a $\mbox{\rm Leb}_{d,d}$-nullset, and so is $\cR:=
\bigcup_{\ell = 1}^q \cR_{\ell}$. Whenever $(x,y)\not \in \cR$, it
follows from (\ref{eq4t1p2}) that
$$
\log_b |x^\top A^n y| = n \log_b r_{\sigma}(A) + k_{\max}\log_b n
+ \log_b |x^\top B_n y + x^\top E_n y|
$$
for all sufficiently large $n$, and since $\log_b r_{\sigma}(A)$ is
rational and $(x^\top B_n y)$ is periodic, Lemma \ref{lem240} shows
that $(x^\top A^n y)$ is not $b$-Benford. In other words,
$\B_b(A) \subset \cR$, so in particular $\R^d \times \R^d \setminus
\B_b(A)$ is not a nullset, i.e., (i) fails.

It remains to consider the case of $L_0\ge 1$. In this case,
label the elements of $\sigma_{EP}^+(A)$ as $\lambda_1 , \ldots ,
\lambda_L$ with $L\ge L_0$ and assume w.l.o.g.\ that 
the $L_0 + 1$ numbers $1, \frac1{2\pi} \arg
\lambda_1 , \ldots , \frac1{2\pi} \arg \lambda_{L_0}$ are
$\Q$-independent. Given any $u\in \R^{L_0}$, there exist $x_u, y_u \in
\R^d$ such that
\begin{align}\label{eq4t1p3}
x_u^\top A^n y_u & = r_{\sigma}(A)^n n^{k_{\max}} \left( 
\sum\nolimits_{\ell = 1}^{L_0} u_{\ell} \cos (n \arg \lambda_{\ell}) +
z_n 
\right) \nonumber \\ 
& =
r_{\sigma}(A)^n n^{k_{\max}} \left( 
\Re \left(
\sum\nolimits_{\ell =1}^{L_0} u_{\ell} e^{\imath n \arg \lambda_{\ell}}
\right) + z_n 
\right)
 \, , \quad \forall n \in \N \, ,
\end{align}
where $(n z_n)$ is a bounded sequence in $\R$. On the other hand,
(\ref{eq4t1p1}) implies
\begin{equation}\label{eq4t1p4}
x_u^\top A^n y_u =  r_{\sigma}(A)^n n^{k_{\max}} \left( 
\Re \left(
\sum\nolimits_{\lambda\in \sigma_{EP}^+(A)} x_u^\top C_{\lambda} y_u
e^{\imath n \arg \lambda}
\right) + x_u^\top E_n y_u
\right) \, , \!\! \! \quad \forall n \in \N \, .
\end{equation}
Comparing (\ref{eq4t1p3}) and (\ref{eq4t1p4}) yields
$$
\Re \left(
\sum\nolimits_{\ell = 1}^{L_0} u_{\ell} e^{\imath n \arg
  \lambda_{\ell}}
-
\sum\nolimits_{\lambda \in \sigma_{EP}^+ (A)} x_u^\top C_{\lambda} y_u
e^{\imath n \arg \lambda}
\right) \: \stackrel{n\to \infty}{\longrightarrow} \: 0 \, .
$$
Lemma \ref{cor34c} shows that $x_u^\top C_{\lambda_{\ell}} y_u =
u_{\ell}$ for every $\ell \in \{ 1, \ldots , L_0\}$, and $x_u^\top
C_{\lambda_{\ell}} y_u =0$ for every $\ell \in \{L_0 + 1 , \ldots ,
L\}$. Recall now that $\log_b r_{\sigma(A)}\in \mbox{\rm
  span}_{\Q}\Delta_{\sigma_{EP}^+(A)}$. Lemma \ref{lem2Omega} guarantees
that it is possible to choose $u\in \R^{L_0}$ in such a way that the
sequence $(x_u^\top A^n y_u)$ in (\ref{eq4t1p3}) is neither
$b$-Benford nor terminating. The continuity of the map
$$
\left\{
\begin{array}{ccl}
\R^d \times \R^d & \to & \C^L \\
(x,y) & \mapsto &  (x^\top C_{\lambda_1} y , \ldots , x^\top
C_{\lambda_L} y) 
\end{array}
\right.
$$
implies that $(x^\top A^n y)$ is not $b$-Benford whenever $x$ and $y$
are sufficiently close to $x_u$ and $y_u$, respectively. Thus $\R^d
\times \R^d \setminus \B_b(A)$ contains a non-empty open set, and so
again (i) fails. This completes the proof of (i)$\Rightarrow$(ii).

To establish the reverse implication (ii)$\Rightarrow$(i), let
$\sigma_{EP}(A^N)$ be $b$-nonresonant and fix any $m\in \{1, \ldots ,
N\}$. It follows from (\ref{eq4t1p1}) that 
$$
A^{nN+m} = r_{\sigma}(A)^{nN} n^{k_{\max}} \left(
\Re \left(
\sum\nolimits_{\lambda \in \sigma^{++}} C_{\lambda} e^{\imath n \arg
  \lambda} 
\right) + E_n
\right) \, , \quad \forall n \in \N \, ,
$$
where $\sigma^{++} \subset \sigma_{EP}^+(A^N)$ is non-empty,
$C_{\lambda} \in \C^{d\times d} \setminus \{0\}$ for every $\lambda
\in \sigma^{++}$, and $(n|E_n|)$ is bounded. (Once again it should be
noted that the set $\sigma^{++}$, the matrices $C_{\lambda}$ and the
sequence $(E_n)$ may all vary with $m$.) The set
$$
\cR_{m,\lambda}:= \bigl\{
(x,y) \in \R^d \times \R^d : x^\top C_{\lambda} y = 0 
\bigr\}
$$
is a $\mbox{\rm Leb}_{d,d}$-nullset, and so is $\cR:= \bigcup_{m=1}^N
\bigcup_{\lambda \in \sigma^{++}} \cR_{m, \lambda}$. Whenever
$(x,y)\not \in \cR$, an argument completely analogous to the one
establishing (ii)$\Rightarrow$(i) in Theorem \ref{thm34} shows that
$(x^\top A^n y)$ is $b$-Benford. Thus $\R^d \times \R^d \setminus
\B_b(A) \subset \cR$, and the proof is complete.
\end{proof}

\begin{example}\label{ex42}
(i) As seen in Example \ref{ex39}(iii) the matrix $A = \frac12
\left[
\begin{array}{cc} 1 + \pi & 1 - \pi \\ 1 - \pi & 1 + \pi
\end{array}
\right]$
has $\sigma(A)=\{1,\pi\}$ $b$-resonant for every $b$. However,
$\sigma_{EP}(A) = \{\pi\}$ is $b$-nonresonant, and since
\begin{equation}\label{eq43}
A^n = \frac{\pi^n}{\pi - 1} (A- I_2) + \frac1{\pi - 1} (\pi I_2 - A)
\end{equation}
for all $n\in \N$, it is clear that 
\begin{align*}
\R^2 \times \R^2 \setminus \B_b
(A) & = \bigl\{ (x,y) \in
\R^2 \times \R^2 : x^\top (A-I_2)y = 0 \bigr\}\\
& =\bigl\{ (x,y) \in \R^2 \times \R^2 : (x_1 - x_2) (y_1 - y_2) = 0 \bigr\}
\end{align*} 
is a nullset. Also, $(|A^n x|)$ is Benford unless $Ax = x$.

(ii) Let $B:= A^{-1}$. Then $\sigma(B) = \{\pi^{-1}, 1\}$, so
$\sigma_{EP}(B^n) =\{1\}$ is $b$-resonant for every $b$ and $n\in \N$. Since
(\ref{eq43}) actually holds for all $n\in \Z$, the sequence $(x^\top
B^n y)$ can only be $b$-Benford if $x^\top (\pi I_2 - A)y=0$, i.e.
\begin{align*}
\B_b (B) & \subset \bigl\{ (x,y)\in \R^2 \times \R^2 : x^\top (\pi I_2 - A)y
= 0\bigr\}\\
& = \bigl\{ (x,y) \in \R^2 \times \R^2 : (x_1 + x_2) (y_1 + y_2) = 0
\bigr\} \, , 
\end{align*}
showing that $\B_b (B)$ is a nullset in this case.
Similarly, $(|B^n x|)$ can only be Benford if $Bx=\pi^{-1} x$.
\end{example}

\begin{rem}\label{rem42a}
Recall from Theorem \ref{thm350} that $(|A^nx|)$ is $b$-Benford
provided that $\sigma(A^N)$ is $b$-nonresonant for some $N\in
\N$. If $A$ is not nilpotent then $(|A^nx|)$ is terminating only if
$x$ is an element of the proper subspace (and hence nullset)
$\mbox{\rm ker} A^d$. For {\em almost\/} all $x\in \R^d$, therefore,
$(|A^n x|)$ is $b$-Benford. As it turns out, a much weaker assumption
suffices to guarantee the latter conclusion: Similarly to Theorem \ref{thm41}, it can be shown that
$b$-nonresonance of $\sigma_{EP}(A^N)$ for some $N$ implies
that $(|A^n x|)$ is $b$-Benford for almost all $x\in
\R^d$. Unlike in Theorem \ref{thm41} (yet much like in Theorem
\ref{thm350}), the converse does not
hold in general. In fact, as demonstrated already by Example \ref{exa314},
it is impossible to characterize the $b$-Benford property of $(|A^n
x|)$ for almost all $x\in \R^d$ using only $\sigma (A)$, let alone $\sigma_{EP}(A)$.
\end{rem}

The following variant of Theorem \ref{thm380} is motivated by
Theorem \ref{thm41}. Recall that $\cZ^n = \{z^n : z \in \cZ\}$ for any
$\cZ \subset \C$. If $p=p(z)$ is a non-constant polynomial and $\cZ =
\{z\in \C : p(z) = 0\}$, let $\zeta := \max_{z\in \cZ}|z|$
and, for each $z\in \cZ$, let $k(z)$ be the multiplicity of $z$ as a
root of $p$, that is, $k(z) = \min \{n\in \N : p^{(n)}(z) \ne 0\}$. In
analogy to the extremal peripheral spectrum, define
$$
\cZ_{EP} : = \{z\in \C : p(z) = 0\}_{EP} := \{z\in \cZ \cap \zeta \mS : k(z)
= k_{\max}\} \, ,
$$
where $k_{\max}  := \max\{k(z) : z\in \cZ \cap \zeta \mS\}$.

\begin{theorem}\label{thm43}
Let $a_1, a_2 , \ldots , a_{d-1}, a_d$ be real numbers with $a_d \ne
0$, and $b\in \N \setminus \{1\}$. Then the following are equivalent:
\begin{enumerate}
\item The solution $(x_n)$ of {\rm (\ref{eq3n1})} is $b$-Benford for
  almost all $(x_1, \ldots , x_d)\in \R^d$;
\item With the polynomial $p(z) = z^d - a_1 z^{d-1} - a_2 z^{d-2} - \ldots
  - a_{d-1}z - a_d$, the set
  $\{z\in \C : p(z) = 0\}_{EP}^N$ is $b$-nonresonant for some $N\in \N$.
\end{enumerate}
\end{theorem}

\begin{proof}
As seen in the proof of Theorem \ref{thm380}, for the matrix $A$
associated with (\ref{eq3n1}) via (\ref{eq3n2}), $\sigma(A) = \{z\in
\C : p(z)=0\}$, and in fact $\sigma_{EP}(A^n) = \{z\in \C:
p(z)=0\}_{EP}^n$ for every $n\in \N$. With this as well as
(\ref{eq3n3}) and (\ref{eq4t1p1}), the argument is completely
analogous to the proof of Theorem \ref{thm41}; details are left to the reader.
\end{proof}

\begin{example}
(i) For convenience let $b=10$ and consider the third-order equation
\begin{equation}\label{eq410}
x_n = 5 x_{n-1} - 11 x_{n-2} + 15 x_{n-3} \, , \quad \forall n \ge 4
\, .
\end{equation}
With the associated set
$$
\cZ = \{z\in \C : z^3 - 5z^2 + 11 z - 15 = 0\} = \{z\in \C : (z-3)(z^2
- 2z + 5) = 0\} \, ,
$$
clearly $\zeta = 3$, and $\cZ_{EP} = \{3\}$ is $b$-nonresonant. For
almost all $(x_1,x_2,x_3)\in \R^3$, therefore, the solution $(x_n)$ of
(\ref{eq410}) is $10$-Benford. In fact, since $\lim_{n\to \infty}
3^{-n} x_n = \frac1{24} (x_3 - 2 x_2 + 5x_1)$, the sequence $(x_n)$ is
$10$-Benford unless $x_3 = 2 x_2 - 5x_1$. Note that in the latter
case, $(x_n)$ solves the {\em second}-order equation $x_n = 2 x_{n-1}
- 5x_{n-2}$, i.e.\ (\ref{eqrec1}) with $\gamma = 1$, and as seen
in Example \ref{ex316b}, except for the trivial case of $x_n \equiv 0$ it is not known
whether $(x_n)$ is $10$-Benford.

(ii) The set $\cZ$ associated with the second-order equation
\begin{equation}\label{eq411}
x_n = \pi^{-2} x_{n-2} \, , \quad \forall n \ge 3 \, ,
\end{equation}
i.e.\ $\cZ = \{\pm \pi^{-1}\}$ is $b$-resonant for all $b\in \N
\setminus \{1\}$. However, with $\zeta = \pi^{-1}$, the set $\cZ^2 =
\cZ_{EP}^2= \{\pi ^{-2}\} $ is $b$-nonresonant. Hence the solution
$(x_n)$ of (\ref{eq411}) is Benford for almost all $(x_1, x_2)\in
\R^2$. Again, it is easy to check that in fact $(x_n)$ is Benford if
and only if $x_1 x_2 \ne 0$.

(iii) As a variant of (\ref{eq411}), consider the recursion
\begin{equation}\label{eq412}
x_n = (1 - \pi^{-2}) x_{n-1} + \pi^{-2} x_{n-2} \, , \quad \forall n
\ge 3 \, .
\end{equation}
Now $\cZ = \{-\pi^{-2}, 1\}$, hence $\zeta = 1$, and $\cZ_{EP}^n =
\{1\}$ is $b$-resonant for every $n\in \N$. By Theorem \ref{thm43},
the solution $(x_n)$ of (\ref{eq412}) is not Benford for almost all
$(x_1, x_2)\in \R^2$. In fact, $(x_n)$ can only be Benford if $x_1 +
\pi^2 x_2 = 0$, hence $\{(x_1,x_2)\in \R^2 : (x_n) \: \mbox{\rm is
  Benford}\}$ is a nullset.
\end{example}

\begin{rem}\label{rem44}
In light of the above examples, it may be conjectured that in the context of Theorem \ref{thm41},
$\B_b(A)$ is actually a nullset if $\sigma_{EP}(A^n)$ is
$b$-resonant for all $n\in \N$. Similarly, the solution $(x_n)$ of (\ref{eq3n1}) may for
almost all $(x_1, \ldots , x_d)\in \R^d$ not be $b$-Benford whenever $\{z\in
\C : p(z) = 0\}_{EP}^n$ is $b$-resonant for every $n$. Using Lemmas
\ref{lema6} and \ref{lema7}, it is not hard to verify this conjecture
for $d\in \{1,2,3\}$. However, the authors do not know of any proof of, or counter-example to
the conjecture for $d\ge 4$; cf.\ Remark \ref{rema9}(i).
\end{rem}

Clearly, if $\sigma(A^N)$ is $b$-nonresonant for some $N\in \N$ then so is
$\sigma_{EP}(A^N)$, and unless $A$ is nilpotent, this in turn implies that
$\log_b r_{\sigma}(A)$ is irrational. As the next result shows, even
the latter, seemingly much weaker condition alone suffices to recover a
strong form of Theorem \ref{thm34} --- provided that some power of $A$
is positive. Recall that $A\in \R^{d\times d}$ is {\em positive\/}
({\em nonnegative\/}), in symbols $A>0$ ($A\ge 0$), if $[A]_{jk}>0$
($[A]_{jk}\ge 0$) for all $j,k\in \{1, \ldots , d\}$; here $[A]_{jk}$
denotes the entry of $A$ at position $(j,k)$, i.e.\ in the $j$-th row
and $k$-th column, thus $[A]_{jk}= (e^{(j)})^{\top} A e^{(k)}$. For
convenience, write $x> 0$ ($x\ge 0$) for $x\in \R^d$ if $x_j > 0$ 
($x_j \ge 0$) for all $j\in \{1, \ldots , d\}$. A proof of the
following result can be found in \cite[Sec.3]{BE} for $b=10$, but the
argument given there immediately carries over to arbitrary base $b$.

\begin{prop}\label{prop45}
Let $A\in \R^{d\times d}$ and $b\in \N \setminus \{1\}$. Assume that
$A^N>0$ for some $N\in \N$. Then the following four statements are
equivalent:
\begin{enumerate}
\item For every $x,y\in \R^d \setminus \{0\}$ with $x\ge 0$, $y\ge 0$
  the sequence $(x^\top A^n y)$ is $b$-Benford;
\item For every $x\in \R^d\setminus \{0\}$ with $x\ge 0$ the sequence
  $(|A^nx|)$ is $b$-Benford;
\item The sequence $(|A^n|)$ is $b$-Benford;
\item $\log_b r_{\sigma}(A)$ is irrational.
\end{enumerate}
\end{prop}

\begin{example}\label{ex46}
(i) For the matrix
$$
A = \left[
\begin{array}{ccc}
0 & 1 & 0 \\ 0 & 0 & 1 \\ 6 & 1 & 0 
\end{array}
\right] \, ,
$$
one finds $\sigma (A) = \{-1 \pm \imath \sqrt{2}, 2\}$, and so $\sigma
(A)$ is $b$-resonant whenever $b\in \{2^n, 3^n : n \in \N\}$. For any
other base $b$, and similarly to Example \ref{ex316b}(iii), it is
apparently unknown whether $\sigma (A)$ is $b$-resonant. Note,
however, that $A\ge 0$ and $A^5>0$, hence Proposition \ref{prop45}
applies with $r_{\sigma}(A)=2$. For every $b$ not an integer power of $2$, therefore, and for all $x,
y \in \R^3 \setminus \{0\}$ with $x,y\ge 0$, the sequences $(x^\top
A^n y)$ and $(|A^n x|)$ are $b$-Benford. This nicely complements the
fact that $(x^\top A^n y)$ and $(|A^n x|)$ are $b$-Benford in this
case for almost all $(x,y)\in \R^3 \times \R^3$ (by Theorem
\ref{thm41}) and almost all $x\in \R^3$ (by Remark \ref{rem42a}),
respectively. Also, $(|A^n|)$ is $b$-Benford by Theorem \ref{thm360}.

(ii) For the matrix
$$
B = \left[
\begin{array}{rcc}
-3 & 1 & 0 \\ 1 & 0 & 1 \\ 0 & 1 & 6 
\end{array}
\right] \, ,
$$
it is easily checked that $B^8 > 0$. An argument similar to, but
simpler than the one in Example \ref{ex39}(ii) shows that
$\log_b r_{\sigma}(B)$ is irrational for every $b$. Hence again Proposition
\ref{prop45} applies. Note that in order to reach this conclusion it is not necessary to explicitly
determine the value of $r_{\sigma}(B)$.
\end{example}

\begin{rem}\label{rem47}
For nonnegative $A\in \R^{d\times d}$, it is well-known that $A^N>0$
for some $N\in \N$ (if and) only if $A^{d^2 - 2d + 2} >0$; see e.g.\ \cite[Prop.8.5]{HJ}. On the
other hand, for $d\ge 3$ and {\em arbitrary\/} $A\in \R^{d\times d}$, the
minmal number $N$ for which $A^N>0$, if at all existant, may be
arbitrarily large; see \cite[Sec.3]{BE}.
\end{rem}

Proposition \ref{prop45} has a counterpart for difference
equations which is a variant of Theorem \ref{thm380} under the
assumption of {\em positivity}, both for the coefficients and the
initial data; for a proof the reader is again referred to \cite{BE}.

\begin{prop}\label{prop46a}
Let $a_1, a_2, \ldots , a_{d-1}, a_d$ be positive real numbers, and
$b\in \N \setminus \{1\}$. Then the following are equivalent:
\begin{enumerate}
\item Every solution $(x_n)$ of {\rm (\ref{eq3n1})} with $x_1, \ldots
  , x_d \ge 0$ and $\max_{j=1}^d x_j > 0$ is $b$-Benford;
\item $\log_b \zeta$ is irrational where $z=\zeta$ is the right-most
  root of $p(z)=0$ with the polynomial $p(z) = z^d - a_{1}z^{d-1} -
  a_2 z^{d-2}  - \ldots - a_{d-1}z - a_d$.
\end{enumerate}
\end{prop}

To finally put Theorems \ref{thm34} and \ref{thm41} as well as
Corollary \ref{cor370} in perspective, recall that, informally put,
$b$-Benford sequences are prevalent among the sequences $(x^\top A^n
y)$, $(|A^n x|)$, and $(|A^n|)$ derived from $(A^n)$ whenever $\sigma (A^N)$
is $b$-nonresonant for some $N\in \N$. For {\em most\/} matrices $A\in \R^{d\times d}$
the set $\sigma (A)$ is $b$-nonresonant for every $b$, as are
$\sigma (A^n)$ and $\sigma_{EP}(A^n)$ for all $n\in \N$, and $\log_b
r_{\sigma}(A)$ is irrational. More formally, let
$$
\cG_{d,b}:= \bigl\{ A\in \R^{d\times d} : A \: \mbox{\rm is invertible and
  $\sigma (A)$ is $b$-nonresonant}\, \bigr\} \, .
$$
With this, it can be shown that while the set $\R^{d\times d}\setminus
\cG_{d,b}$ is dense in $\R^{d\times d}$, it nevertheless is a
first-category set (i.e.\ a countable union of nowhere dense sets) and
has (Lebesgue) measure zero. The same, therefore, is true for
$\bigcup\nolimits_{b\in \N \setminus \{1\}} (\R^{d\times d} \setminus
\cG_{d,b})$. In other words, most real $d\times d$-matrices, both
in a topological and measure-theoretical sense, belong to
$\bigcap_{b\in \N \setminus \{1\}} \cG_{d,b}$, and thus are invertible
with their spectrum $b$-nonresonant for every $b$; see e.g.\ \cite{BDCDSA,
  BHKR,BHPS} for details. This observation may help explain the
conformance to BL often observed empirically across a wide range of
scientific disciplines.

\subsubsection*{Acknowledgements}

The authors have been supported by an {\sc Nserc} Discovery
Grant. They like to thank T.P.\ Hill, B.\ Schmuland, M.\ Waldschmidt, A.\ Weiss and
R.\ Zweim\"{u}ller for helpful discussions and comments.

\begin{appendix}

\section{Some auxiliary results}

\setcounter{equation}{0}
\renewcommand{\theequation}{A.\arabic{equation}}

The purpose of this appendix is to provide proofs for several
analytical facts that have been used in establishing the main results
of this article. Throughout, let $d$ be a fixed positive integer.

\begin{lem}\label{lem34b}
Given any $z_1, \ldots , z_d \in \mS = \{z\in \C : |z| =
1\}$, the following are equivalent:
\begin{enumerate}
\item If $c_1, \ldots , c_d \in \C$ and $\lim_{n\to \infty} (c_1 z_1^n + \ldots + c_d z_d^n)$
  exists then $c_1 = \ldots = c_d = 0$;
\item $z_j \not \in \{1\}\cup \{z_k : k\ne j\}$ for every $1\le j \le d$.
\end{enumerate}
\end{lem}

\begin{proof}
Clearly (i)$\Rightarrow$(ii) because if $z_j = 1$ for some $j$ simply let
$c_j = 1$ and $c_{\ell}=0$ for all $\ell \ne j$, whereas if $z_j =
z_k$ for some $j\ne k$ take $c_j = 1$, $c_k = -1$, and $c_{\ell} = 0$
for all $\ell \in \{1, \ldots , d\}\setminus \{j,k\}$. To show that
(ii)$\Rightarrow$(i) as well, proceed by induction. Trivially, if
$d=1$ then $(c_1 z_1^n)$ with $z_1\in \mS$ converges only if $c_1=0$
or $z_1 = 1$. Assume now that (ii)$\Rightarrow$(i) has been
established already for some $d\in \N$, let $z_1, \ldots, z_{d+1}\in \mS$, and assume that
$z_j \not \in \{1\}\cup \{z_k : k \ne j\}$ for every $1\le j \le
d+1$. If $\lim_{n\to \infty} (c_1 z_1^n + \ldots + c_{d+1} z_{d+1}^n)$
exists then, as $z_{d+1}\ne 1$,
\begin{align*}
& \left\{ 
c_1 \left( \frac{z_1}{z_{d+1}}\right)^n  \frac{z_1 - 1}{z_{d+1} -1 } +
\ldots +
c_d \left( \frac{z_d}{z_{d+1}}\right)^n \frac{z_d - 1}{z_{d+1} -1 }
+ c_{d+1}
\right\} z_{d+1}^n (z_{d+1} - 1) \\[2mm]
& \quad \qquad  = c_1 z_1^n (z_1 - 1) + \ldots + c_{d+1} z_{d+1}^n (z_{d+1} - 1) \\[2mm]
& \quad \qquad  = c_1 z_1^{n+1}  + \ldots + c_{d+1} z_{d+1}^{n+1}
-
\bigl( c_1 z_1^n + \ldots + c_{d+1} z_{d+1}^n \bigr) 
\enspace \stackrel{n\to \infty}{\longrightarrow} \enspace 0 \, ,
\end{align*}
which in turn yields
$$
\lim\nolimits_{n\to \infty} \left\{ 
c_1 \frac{z_1 - 1}{z_{d+1} -1 }  \left( \frac{z_1}{z_{d+1}}\right)^n +
\ldots +
c_d \frac{z_d - 1}{z_{d+1} -1 } \left( \frac{z_d}{z_{d+1}}\right)^n 
\right\} = -c_{d+1} \, .
$$
Note that $\displaystyle \frac{z_j}{z_{d+1}} \not \in \{1\}\cup \left\{
\frac{z_k}{z_{d+1}} : k \ne j\right\}$ for every $1\le j \le d$. By
the induction assumption, $\displaystyle c_j \frac{z_j-1}{z_{d+1} - 1} = 0$ for all
$1\le j \le d$. Hence $c_1 = \ldots = c_{d}=0$, and clearly
$c_{d+1} = 0$ as well.
\end{proof}

Two simple consequences of Lemma \ref{lem34b} have been used
repeatedly.

\begin{lem}\label{cor34c}
Let $0 = t_0 < t_1 < \ldots < t_{d} < t_{d+1} = \pi$ and $c_0, c_1
\ldots , c_{d}, c_{d+1}\in \C$. If 
$$
\lim\nolimits_{n\to \infty} \Re (c_0
e^{\imath n t_0} + c_1 e^{\imath n t_1} + \ldots + c_d e^{\imath n
  t_d} + c_{d+1} e^{\imath n t_{d+1}}) = 0\, ,
$$
then $\Re c_0 = \Re c_{d+1} = 0$ and $c_1 = \ldots = c_d = 0$.
\end{lem}

\begin{proof}
For every $j\in \{1, \ldots , 2d+1\}$ let
$$
z_j = \left\{
\begin{array}{lll}
e^{\imath t_j} & & \mbox{\rm if } 1 \le j \le d+1 \, , \\[1mm]
e^{-\imath t_{2d+2 - j}} & & \mbox{\rm if } d+2 \le j \le 2d+1 \, ,
\end{array}
\right.
$$
and note that $z_j \not \in \{1\} \cup \{z_k : k\ne j\}$. Since
\begin{align*}
2 \lim\nolimits_{n\to \infty} \Re & \left(
c_0 e^{\imath n t_0}  + c_1 e^{\imath n t_1} + \ldots + c_d e^{\imath n
t_d} + c_{d+1} e^{\imath n t_{d+1}}
\right)  - 2 \Re c_0  \\
& = 2  \lim\nolimits_{n\to \infty} \Re \left( c_1 e^{\imath n t_1} + \ldots + c_d e^{\imath n
t_d} + c_{d+1} e^{\imath n t_{d+1}} \right) \\
& = \lim\nolimits_{n\to \infty} \left(
\sum\nolimits_{j=1}^d c_j z_j^n + 2 (\Re c_{d+1} ) z_{d+1}^n +
\sum\nolimits_{j=d+2}^{2d+1} \overline{c_{2d+2 -j}} z_j^n
\right) 
\end{align*}
exists by assumption, Lemma \ref{lem34b} shows that $c_1 = \ldots =
c_d = 0$ and $\Re c_{d+1}=0$, and so clearly $\Re c_0 = 0$ as well.
\end{proof}

\begin{lem}\label{lem34c}
Given any $z_1, \ldots , z_d \in
\mS$, the following are equivalent:
\begin{enumerate}
\item If $c_1, \ldots , c_d \in \C$ and $\lim_{n\to \infty} \Re(c_1 z_1^n + \ldots + c_d z_d^n)$
  exists then $c_1 = \ldots = c_d =
  0$;
\item $z_j \not \in \{-1, 1\}\cup \{z_k, \overline{z_k} : k\ne j\}$ for every $1\le j \le d$.
\end{enumerate}
\end{lem}

\begin{proof}
Clearly (i)$\Rightarrow$(ii) because if $z_j\in \{-1,1\}$ for some $1\le j
\le d$ simply let $c_j = \imath$ and $c_{\ell} = 0$ for all $\ell \ne j$,
whereas if $z_j \in \{z_k , \overline{z_k}\}$ for some $j\ne k$, take
$c_j = 1$, $c_k = -1$, and $c_{\ell} = 0$ for all $\ell \in \{1,
\ldots , d\}\setminus \{j,k\}$.
Conversely, if
$$
\lim\nolimits_{n\to \infty} \Re(c_1 z_1^n + \ldots + c_d z_d^n) =
{\textstyle \frac12}
\lim\nolimits_{n\to \infty} (c_1 z_1^n + \overline{c_1} \, 
\overline{z_1}^n + \ldots + c_d z_d^n + \overline{c_d} \, \overline{z_d}^n)
$$
exists then, by Lemma \ref{lem34b}, $c_1 = \ldots = c_d = 0$ unless
either $z_j = 1$ or $z_j = \overline{z_j}$ (and hence $z_j \in
\{-1,1\}$) for some $j$, or else $z_j \in \{z_k , \overline{z_k}\}$ for
some $j\ne k$. Overall, $c_1 = \ldots = c_d=0$ unless $z_j \in
\{-1,1,z_k , \overline{z_k}\}$ for some $j\ne k$. Thus
(ii)$\Rightarrow$(i), as claimed.
\end{proof}

Let $\vartheta_1, \ldots ,
\vartheta_d$ and $\beta\ne 0$ be real numbers, and $p_1, \ldots ,
p_d$ integers. With
these ingredients, consider the sequence
$(x_n)$ of real numbers given by
\begin{equation}\label{ea1}
x_n = p_1 n \vartheta_1 + \ldots + p_d n \vartheta_d + \beta \ln
\big|u_1 \cos (2\pi n \vartheta_1) + \ldots + u_d \cos (2\pi n 
\vartheta_d) \big| \, , \quad \forall n \in \N \, ,
\end{equation}
where $u\in \R^d$. Recall that Lemma \ref{lem2Omega}, which has been
instrumental in the proof of Theorem \ref{thm34}, asserts that it
is possible to choose $u\in \R^d$ in such a way that $(x_n)$ is {\em
  not\/} u.d.\ mod $1$ whenever the $d+1$ numbers $1,\vartheta_1, \ldots , \vartheta_d$
are $\Q$-independent. The remainder of this appendix is devoted to
providing a rigorous proof of Lemma \ref{lem2Omega}.

To prepare for the argument, recall that $\T^d$ denotes the $d$-dimensional
torus $\R^d/\Z^d$, together with the $\sigma$-algebra $\cB(\T^d)$ of
its Borel sets. Let $\cP(\T^d)$ be the set of all probability
measures on $\bigl( \T^d , \cB (\T^d)\bigr)$, and given any $\mu \in
\cP(\T^d)$, associate with it the family $\bigl( \widehat{\mu}
(k)\bigr)_{k\in \Z^d}$ of its {\em Fourier coefficients}, defined as
$$
\widehat{\mu} (k) = \int_{\T^d} e^{2\pi \imath k^{\top} t} \, {\rm d} \mu (t)
=
\int_{\T^d} e^{2\pi \imath (k_1 t_1 + \ldots + k_d t_d)} \,  {\rm d} \mu (t_1,
\ldots , t_d) \, , \quad \forall k\in \Z^d \, .
$$
Recall that $\mu \mapsto \bigl( \widehat{\mu}
(k)\bigr)_{k\in \Z^d}$ is one-to-one, i.e., the Fourier coefficients
determine $\mu$ uniquely. Arguably the most prominent element in
$\cP(\T^d)$ is the Haar measure $\lambda_{\T^d}$
for which, with ${\rm d}\lambda_{\T^d}(t)$ abbreviated ${\rm d}t$ as usual,
$$
\widehat{\lambda_{\T^d}}(k) = \int_{\T^d} e^{2\pi \imath (k_1 t_1 + \ldots
  + k_d t_d)} \, {\rm d}t = \prod\nolimits_{j=1}^d \int_{\T} e^{2\pi
  \imath k_j t} {\rm d}t = \left\{
\begin{array}{cll}
1 & & \mbox{if }k=0 \in \Z^d \, , \\
 0 & & \mbox{if } k \ne  0 \, .
\end{array}
\right.
$$
Given $\mu \in \cP (\T^d)$, therefore, to show that $\mu \ne
\lambda_{\T^d}$ it is (necessary and) sufficient to find at least one
$k \in \Z^d \setminus \{0\}$ for which $\widehat{\mu} (k)\ne 0$.
Recall also that, given any (Borel) measurable map $T:\T^d \to \T$,
each $\mu \in \cP (\T^d)$ induces a unique $\mu
\circ T^{-1}\in \cP (\T)$, via
$$
\mu \circ T^{-1} (B) = \mu \bigl( T^{-1}(B)\bigr) \, , \quad \forall B \in
\cB(\T) \, .
$$
Note that the Fourier coefficients of $\mu \circ T^{-1}$ are simply
$$
\widehat{ \mu\circ T^{-1}}(k) = \int_{\T} e^{2\pi \imath k t} {\rm
  d}(\mu \circ T^{-1}) (t) = \int_{\T^d} e^{2\pi \imath k T(t)} {\rm
  d}\mu (t) \, , \quad k \in \Z \, .
$$
If in particular $d=1$ and $\mu \circ T^{-1} = \mu$ then $\mu$ is said
to be $T$-{\em invariant\/} (and $T$ is $\mu$-{\em
  preserving\/}).

With a view towards Lemma \ref{lem2Omega},
for any $p_1, \ldots , p_d \in \Z$ and $\beta \in
\R$ consider the map
\begin{equation}\label{ea2}
\Lambda_u : \left\{
\begin{array}{ccl}
\T^d & \to & \T \, , \\[1mm]
t  & \mapsto & \bigl\langle p_1 t_1 + \ldots + p_d t_d + \beta \ln \big|
u_1 \cos (2\pi t_1) + \ldots + u_d \cos (2\pi t_d)\big| \bigr\rangle \, ;
\end{array}
\right.
\end{equation}
here $u \in \R^d$ may be thought of as a parameter. (Recall the
convention, adhered to throughout, that $\ln 0 = 0$.) Note that
each map $\Lambda_u$ is (Borel) measurable, in fact differentiable
outside a set of $\lambda_{\T^d}$-measure zero. For
every $\mu \in \cP(\T^d)$, therefore, the measure $\mu \circ \Lambda_u^{-1}$ is a
well-defined element of $\cP(\T)$.
Lemma \ref{lem2Omega} is a consequence of the following fact which
may also be of independent interest.

\begin{theorem}\label{lemx}
For every $p_1, \ldots , p_d \in \Z$ and $\beta \in \R \setminus \{0\}$,
there exists $u\in \R^d$ such that $\lambda_{\T^d} \circ \Lambda_u^{-1} \ne
\lambda_{\T}$, with $\Lambda_u$ given by {\rm (\ref{ea2})}.
\end{theorem}

To see that Theorem \ref{lemx} does indeed imply Lemma \ref{lem2Omega},
let $p_1, \ldots , p_d \in \Z$ and $\beta \in \R \setminus \{0\}$ be
given, and pick $u\in \R^d$ such that $\lambda_{\T^d} \circ
\Lambda_u^{-1} \ne \lambda_{\T}$. Consequently, there exists a continuous
function $f:\T \to \C$ for which $\int_{\T} f \, {\rm d}(
\lambda_{\T^d}\circ \Lambda_u^{-1} ) \ne \int_{\T} f \, {\rm d}\lambda_{\T}$. Note
that $f\circ \Lambda_u: \T^d \to \C$ is continuous $\lambda_{\T^d}$-almost
everywhere as well as bounded, hence Riemann integrable. Also recall
that the sequence $\bigl(   (n\vartheta_1, \ldots ,
n\vartheta_d) \bigr)$
is u.d.\ mod $1$ in $\R^d$ whenever $1, \vartheta_1, \ldots ,
\vartheta_d$ are $\Q$-independent \cite[Exp.I.6.1]{KN}. In the latter
case, therefore, 
\begin{align*}
\lim\nolimits_{N\to \infty} \frac1{N} \sum\nolimits_{n=1}^N f(\langle
x_n \rangle) &
= \lim\nolimits_{N\to \infty} \frac1{N} \sum\nolimits_{n=1}^N  f \circ
\Lambda_u \bigl( \langle (n\vartheta_1, \ldots , n\vartheta_d) \rangle
\bigr) \\
& = \int_{\T^d} f \circ \Lambda_u \, {\rm d}\lambda_{\T^d} =
\int_{\T} f \, {\rm d}(\lambda_{\T^d} \circ \Lambda_u^{-1} ) \ne
\int_{\T} f \, {\rm d}\lambda_{\T} \, , 
\end{align*}
showing that $(x_n)$ is {\em not\/} u.d.\ mod $1$.

Thus it remains to prove Theorem \ref{lemx}. Though the assertion of the
latter is quite plausible intuitively, the authors do not know of any
simple but rigorous justification. The proof presented here is
computational and proceeds
in essentially two steps: First the case of $d=1$ is analyzed in
detail. Specifically, it is shown that $\lambda_{\T} \circ \Lambda_u^{-1}
\ne \lambda_{\T}$ unless $p_1 \ne 0$ and $\beta u_1 = 0$. For itself,
this could be seen directly by noticing that the map $\Lambda_u :\T \to \T$ has a
non-degenerate critical point whenever $\beta u_1 \ne 0$, and hence cannot possibly preserve
$\lambda_{\T}$, see e.g.\ \cite[Lem.2.6]{BE} or
\cite[Ex.5.27(iii)]{BHPS}. The more elaborate calculation given here, however, is
useful also in the second step of the proof, i.e.\ the analysis for
$d\ge 2$. As it turns out, the case of $d\ge 2$ can, in essence, be reduced
to calculations already done for $d=1$.

To concisely formulate the subsequent
results, recall that the Euler Gamma function, denoted $\Gamma = \Gamma
(z)$ as usual, is a meromorphic function with poles precisely
at $z\in - \N_0 = \{ 0,-1,-2, \ldots\}$, and $\Gamma (z+1) = z \Gamma (z)\ne 0$ for every $z\in \C
\setminus (-\N_0)$. Also, for convenience every ``empty
sum'' is understood to equal zero, e.g.\ $\sum_{2\le j \le 1} j^2 =0$, whereas
every ``empty product'' is understood to equal $1$, e.g.\ $\prod_{2\le
j \le 1}
j^2 = 1$. Finally, the standard (ascending) Pochhammer symbol $(z)_n$
will be used where, given any $z\in \C$,
$$
(z)_n:= z (z+1) \ldots (z+n-1) = \prod\nolimits_{\ell =0}^{n-1}
(z+\ell) \, , \quad \forall n \in \N \, , 
$$
and $(z)_0:= 1$, in accordance with the convention on empty
products. Note that $(z)_n = \Gamma (z+n)/\Gamma (z)$ whenever $z\not
\in \C \setminus (-\N_0)$.

For every $p\in \Z$ and $\beta \in \R$, consider now the integral
\begin{equation}\label{ea3}
I_{p,\beta}:= \int_{\T} e^{4\pi \imath p t + 2\imath \beta \ln |\cos (2\pi
  t)|} \, {\rm d}t \, .
\end{equation}
The specific form of $I_{p, \beta}$ is suggested by the Fourier
coefficients of $\lambda_{\T} \circ \Lambda_u^{-1}$ in the case of $d=1$; see
the proof of Lemma \ref{lema4} below. Not surprisingly, the value of
$I_{p,\beta}$ can be expressed explicitly by means of special functions.

\begin{lem}\label{lema2}
For every $p\in \Z$ and $\beta \in \R \setminus \{0\}$,
\begin{equation}\label{ea4}
I_{p,\beta} =  
 (-1)^p e^{-\imath \beta \ln 4} \frac{2\imath \beta \Gamma (2\imath \beta)}{\bigl(
  \imath \beta \Gamma (\imath \beta) \bigr)^2} \cdot
\frac{(-\imath \beta)_{|p|}}{(1+\imath \beta)_{|p|}} \, ,
\end{equation}
and hence in particular
\begin{equation}\label{ea5}
|I_{p,\beta}|^2 = \frac{\beta \tanh (\pi \beta)}{ \pi ( p^2 + \beta^2)
} > 0 \, .
\end{equation}
\end{lem}

\begin{proof}
Substituting $-t$ for $t$ in (\ref{ea3}) shows that $I_{p,\beta} =
I_{|p|, \beta}$, and a straightforward calculation, with $T_{\ell}$
denoting the $\ell$-th Chebyshev polynomial ($\ell \in \N_0$), yields
\begin{align*}
I_{p,\beta} & = \int_{\T} e ^{4\pi \imath |p|t + 2 \imath \beta \ln |\cos (2\pi
  t)|} \, {\rm d}t  = \int_0^1 e^{2\pi \imath |p| x + 2 \imath
  \beta \ln |\cos (\pi x)|} \, {\rm d}x \\[2mm]
& = \int_0^{\frac12} 2 \cos (2\pi |p| x ) e^{2\imath \beta \ln |\cos (\pi
  x)|} \, {\rm d}x = 2 \int_0^{\frac12} T_{2|p|} \bigl(\cos(\pi
x)\bigr) e^{2\imath \beta \ln |\cos (\pi x)|} \, {\rm d}x \\[2mm]
& = \frac2{\pi} \int_0^1 \frac{T_{2|p|}(x)}{\sqrt{1-x^2}} e^{2\imath \beta
  \ln x} \, {\rm d}x = \frac{2}{\pi} \int_0^{+\infty} T_{2|p|} \left(
  \frac1{\sqrt{1+x^2}}\right) \frac{e^{-\imath \beta \ln (1+x^2)}}{1+x^2}
\, {\rm d} x \, .
\end{align*}
As the polynomial $T_{2|p|}$ can, for every $p\in \Z$ and $y\ne 0$, be
written as
$$
T_{2|p|}(y) = y^{2|p|} \sum\nolimits_{\ell = 0}^{|p|}
\left( \!\! \begin{array}{c} 2|p| \\ 2 \ell \end{array} \!\! \right)
(1-y^{-2})^{\ell} \, ,
$$
it follows that 
\begin{align*}
I_{p,\beta} & = \frac2{\pi} \sum\nolimits_{\ell = 0}^{|p|} (-1)^{\ell}
\left( \!\! \begin{array}{c} 2|p| \\ 2 \ell \end{array} \!\! \right)
\int_0^{+\infty} \frac{x^{2\ell}}{(1+x^2)^{1+|p| + \imath \beta}} \, {\rm d}
x \\[2mm]
& = \frac1{\pi}   \sum\nolimits_{\ell = 0}^{|p|} (-1)^{\ell}
\left( \!\! \begin{array}{c} 2|p| \\ 2 \ell \end{array} \!\! \right) 
\int_0^{+\infty} \frac{x^{\ell - \frac12}}{ (1+x)^{1+|p| + \imath \beta}} \,
{\rm d}x \\[2mm]
& = \frac1{\pi \Gamma (1 + |p| + \imath \beta)}  \sum\nolimits_{\ell = 0}^{|p|} (-1)^{\ell}
\left( \!\! \begin{array}{c} 2|p| \\ 2 \ell \end{array} \!\! \right)
\Gamma (\textstyle{\frac12} + \ell ) \Gamma (\textstyle{\frac12} + |p| - \ell + \imath \beta) \, .
\end{align*}
Note that $\Gamma$ is finite and non-zero for each argument appearing in this
sum. Recall that
$$
\Gamma ({\textstyle \frac12} + \ell ) = \frac{(2\ell)! \sqrt{\pi}}{\ell
  ! \,  2^{2\ell} } \, ,
\quad \forall \ell \in \N_0 \, ,
$$
and so
\begin{align*}
I_{p,\beta} & = \frac{(-1)^p (2|p|)!} { \sqrt{\pi} 2^{2|p|} \Gamma
  (1+|p| + \imath \beta)} \sum\nolimits_{\ell=0}^{|p|} \left\{
(-1)^{\ell} \: \frac{2^{2\ell} \Gamma (\frac12 + \ell + \imath
  \beta)}{(2\ell)! (|p| - \ell)!} \right\} \\[2mm]
& = \frac{(-1)^p \Gamma (\frac12 + |p|) \Gamma (\frac12 + \imath
  \beta)}{\pi \Gamma (1+|p| + \imath \beta)} \sum\nolimits_{\ell=0}^{|p|} \left\{
(-1)^{\ell} \! \left( \!\! \begin{array}{c} |p| \\  \ell \end{array} \!\!
\right) \prod\nolimits_{k=1}^{\ell} \frac{2k -1 + 2\imath \beta}{2k -1}
\right\} \\[2mm]
& = \frac{(-1)^p \Gamma (\frac12 + \imath \beta)}{\sqrt{\pi} 2^{|p|} \Gamma
(1+|p| + \imath \beta)}  \cdot \\
& \qquad \qquad \cdot  \sum\nolimits_{\ell=0}^{|p|} \left\{
(-1)^{\ell} \! \left( \!\! \begin{array}{c} |p| \\  \ell \end{array} \!\!
\right) \prod\nolimits_{k=1}^{\ell} (2k - 1 + 2\imath \beta)
\prod\nolimits_{k= \ell + 1}^{|p|} (2k-1)
\right\} \\[2mm]
& =  \frac{(-1)^p \Gamma (\frac12 + \imath \beta)}{\sqrt{\pi} 2^{|p|} \Gamma
(1+|p| + \imath \beta)}  R_{|p|} (2\imath \beta) \, ,
\end{align*}
where, for every $m\in \N_0$, the polynomial $R_m$ is given by
\begin{equation}\label{ea5a}
R_m(z) =  \sum\nolimits_{\ell=0}^{m} \left\{
(-1)^{\ell} \! \left( \!\! \begin{array}{c} m \\  \ell \end{array} \!\!
\right) \prod\nolimits_{k=1}^{\ell} (2k - 1 + z)
\prod\nolimits_{k= \ell + 1}^{m} (2k-1)
\right\} \, .
\end{equation}
Thus for example
$R_0 (z) \equiv 1$, $R_1(z) = - z$, $R_2(z) = -2z + z^2$.
Note that the degree of $R_m$ equals $m$, and for every $m\in \N$ and $j\in
\{0,1,\ldots , m-1\}$,
\begin{align*}
R_m(2j) & =  \sum\nolimits_{\ell=0}^{m} \left\{
(-1)^{\ell} \! \left( \!\! \begin{array}{c} m \\  \ell \end{array} \!\!
\right) \prod\nolimits_{k=1}^{\ell} (2k - 1 + 2j)
\prod\nolimits_{k= \ell + 1}^{m} (2k-1)
\right\} \\[2mm]
& = \sum\nolimits_{\ell=0}^{m} \left\{
(-1)^{\ell} \! \left( \!\! \begin{array}{c} m \\  \ell \end{array} \!\!
\right) \prod\nolimits_{k=j+1}^{m} (2k - 1)
\prod\nolimits_{k= \ell + 1}^{\ell + j} (2k-1)
\right\} \\[2mm]
& = \left\{ \prod\nolimits_{k=j+1}^{m} (2k-1) \right\} \sum\nolimits_{\ell=0}^{m} \left\{
(-1)^{\ell} \! \left( \!\! \begin{array}{c} m \\  \ell \end{array} \!\!
\right) \prod\nolimits_{k=1}^{j} (2\ell + 2k -1)
\right\} \: = 0 \, .
\end{align*}
Here the elementary fact has been used that
$ \sum\nolimits_{\ell=0}^{m} 
(-1)^{\ell} \! \left( \!\! \begin{array}{c} m \\  \ell \end{array} \!\!
\right) Q(\ell) = 0$
holds for every polynomial $Q$ of degree less than $m$. As the
polynomial $R_m$ has
degree $m$, it cannot have any further roots besides $0,2,4, \ldots , 2m-2$,
and so
\begin{equation}\label{ea5aa}
R_m(z) = c_m \prod\nolimits_{\ell = 0}^{m-1} (z - 2\ell) \, , 
\end{equation}
with a constant $c_m$ yet to be determined. The correct value of $c_m$
is readily found by observing that (\ref{ea5aa}) yields
$$
R_m(-1) = c_m \prod\nolimits_{\ell = 0}^{m-1} ( -1 -2\ell) = c_m
(-1)^m \cdot 1 \cdot 3 \cdot \ldots \cdot (2m-1) \, ,
$$
whereas, by the very definition (\ref{ea5a}) of $R_m$,
$$
R_m(-1) = \sum\nolimits_{\ell=0}^{m} \! \left\{
(-1)^{\ell} \! \left( \!\! \begin{array}{c} m \\  \ell \end{array} \!\!
\right) \! \prod\nolimits_{k=1}^{\ell} (2k -2)
\prod\nolimits_{k= \ell + 1}^{m} (2k-1)
\right\}  = \prod \nolimits_{k=1}^m (2k-1) \, .
$$
Thus $c_m = (-1)^m$, and overall
$$
R_m(z) = (-1)^m \prod\nolimits_{\ell = 0}^{m-1} (z- 2\ell) =
\prod\nolimits_{\ell = 0}^{m-1} (2\ell - z) = 2^m
\bigl( -\textstyle{\frac12} z \bigr)_m \, .
$$
With this, one obtains
\begin{align*}
I_{p,\beta} &= \frac{(-1)^p \Gamma (\frac12 + \imath \beta)}{\sqrt{\pi}
  2^{|p|} \Gamma (1+|p| + \imath \beta)} \prod \nolimits_{\ell = 0}^{|p|-1}
(2\ell - 2\imath \beta)  \\
& = \frac{2 (-1)^{p+1} e^{-\imath \beta \ln 4}
}{|p| - \imath \beta}\cdot \frac{\Gamma (2\imath \beta)}{\Gamma
  (\imath \beta)^2} 
\prod\nolimits_{\ell=1}^{|p|} \frac{\ell - \imath \beta}{\ell + \imath
  \beta}  \\[2mm]
& = (-1)^p e^{-\imath \beta \ln 4}
\frac{2\imath \beta \Gamma (2\imath \beta)}{\bigl(
  \imath \beta \Gamma (\imath \beta) \bigr)^2} \cdot
\frac{(-\imath \beta)_{|p|}}{(1+\imath \beta)_{|p|}} 
\, , 
\end{align*}
where the so-called Legendre duplication formula for the
$\Gamma$-function has been used in the form
$$
\Gamma (\imath \beta) \Gamma (\textstyle{\frac12} + \imath \beta) = 2^{1 - 2\imath \beta}
\sqrt{\pi} \, \Gamma (2 \imath \beta) \, , \quad \forall \beta \in \R
\setminus \{0\}\, .
$$
Thus (\ref{ea4}) has been established, and 
together with the standard fact 
$$
|\Gamma (\imath \beta)|^2 = \frac{\pi}{\beta \sinh (\pi \beta)} \, , \quad
\forall \beta \in \R \setminus \{0\} \, ,
$$
this immediately yields
$$
|I_{p,\beta}|^2  = 
\frac{4}{p^2 +\beta^2} \cdot \frac{|\Gamma (2\imath \beta)|^2}{|\Gamma
  (\imath \beta)|^4}
=
\frac{4\beta^2 \pi}{2\beta \sinh (2\pi \beta)} \cdot
\frac{\sinh^2 (\pi  \beta)}{\pi^2 (p^2 + \beta ^2)} = 
\frac{\beta \tanh (\pi \beta)}{\pi ( p^2 + \beta^2)} \, ,
$$
i.e., (\ref{ea5}) holds as claimed.
\end{proof}

An immediate consequence of Lemma \ref{lema2} is that for $d=1$ the map $\Lambda_u$
does typically not preserve $\lambda_{\T}$. Notice that the following
result is much stronger than (and hence obviously proves) Theorem \ref{lemx} for $d=1$.

\begin{lem}\label{lema4}
Let $p_1 \in \Z$, $\beta \in \R$ and $u_1 \in \R$. Then $\lambda_{\T}
\circ \Lambda_u^{-1} = \lambda_{\T}$, where $\Lambda_u$ is given by {\rm (\ref{ea2})} with
$d=1$, if and only if $p_1 \ne 0$ and $\beta u_1 = 0$.
\end{lem}

\begin{proof}
Simply note that for $\beta u_1 = 0$ and every $k\in \Z$,
$$
\widehat{\lambda_{\T} \circ \Lambda_u^{-1}} (k) = \left\{
\begin{array}{cll}
1 & & \mbox{\rm if } kp_1 = 0 \, , \\
0 & & \mbox{\rm if } kp_1 \ne 0 \, ,
\end{array}
\right.
$$
and hence $\lambda_{\T} \circ \Lambda_u^{-1} = \lambda_{\T}$ precisely if
$p_1 \ne 0$. On the other hand, for $\beta u_1 \ne 0$,
\begin{align*}
\widehat{\lambda_{\T} \circ \Lambda_u^{-1}} (2)  & =  \int_{\T}
e^{4\pi \imath t} \, {\rm d}(\lambda_{\T} \circ \Lambda_u^{-1} ) (t)  = 
\int_{\T} e^{4\pi \imath (p_1 t + \beta \ln |u_1 \cos (2\pi t)|)} \,
{\rm d}t  \\
& = e^{4\pi \imath \beta \ln|u_1|} I_{p_1, 2\pi \beta} \ne 0 \, ,
\end{align*}
showing that $\lambda_{\T}\circ \Lambda_u^{-1} \ne \lambda_{\T}$ in this case.
\end{proof}

As indicated earlier, the case of $d\ge 2$ of Theorem \ref{lemx} is now
going to be studied and, in a
way, reduced to the case of $d=1$. To this end, let again $p\in \Z$ and
$\beta \in \R$ be given, and consider the function
$i_{p,\beta}:\R \to \C$ with
\begin{equation}\label{funca}
i_{p,\beta} (x) = \int_{\T} e^{4\pi \imath p t + 2 \imath \beta \ln
  |x + \cos (2\pi t) |} \, {\rm d}t \, , \quad
\forall x \in \R \, .
\end{equation}
A few elementary properties of $i_{p,\beta}$ are contained in

\begin{lem}\label{lema5}
For every $p\in \Z$ and $\beta\in \R$, the function
$i_{p,\beta}$ is continuous and even, with
$|i_{p,\beta}(x)|\le 1$ for all $x\in \R$. Moreover, $i_{p,\beta}(0) = I_{p,\beta}$ and $i_{p,\beta}(1) = e^{\imath \beta
  \ln 4} I_{2p,2\beta}$; in particular, $i_{p,\beta}(0)\ne
i_{p,\beta}(1)$ whenever $\beta \ne 0$.
\end{lem}

\begin{proof}
Since for every $x\in \R$,
$$
\lim\nolimits_{y \to x} \ln| y + \cos (2\pi t) | = \ln | x + \cos (2\pi
t) |
$$
holds for all but (at most) two $t\in \T$, the continuity of
$i_{p,\beta}$ follows from the Dominated Convergence
Theorem. Clearly, $i_{p,\beta}$ is even, with $|i_{p,\beta}(x)| \le
\int_{\T} 1\, {\rm  d}\lambda_{\T} = 1$ for every $x\in \R$, and $i_{p,\beta}(0) =I_{p,\beta}$. 
Finally, it follows from
$$
i_{p,\beta} (1)  = e^{\imath \beta \ln 4}
\int_{\T} e^{4\pi \imath p t + 4 \imath \beta \ln |\cos (\pi t)|} \, {\rm
  d}t = e^{\imath \beta \ln 4} I_{2 p , 2 \beta} \, ,
$$
and (\ref{ea5}) that, for every $p\in \Z$ and $\beta \in \R \setminus \{0\}$,
$$
\left|
\frac{i_{p,\beta}(1)}{i_{p,\beta}(0)}
\right|^2  
= 
\frac{|I_{2p, 2\beta}|^2}{|I_{p,\beta}|^2}
= \frac{2\beta \tanh (2\pi \beta )}{4 p^2 + 4 \beta^2}
\cdot \frac{p^2 + \beta ^2}{\beta \tanh (\pi \beta)} = \frac12 \left(
  1 + \frac1{\cosh (2\pi  \beta)}\right)  < 1 \, ,
$$
and hence $i_{p,\beta}(1)\ne i_{p,\beta}(0)$.
\end{proof}

The subsequent analysis crucially depends on the fact that
$i_{p,\beta}$ is actually much smoother than Lemma
\ref{lema5} seems to suggest. Recall that a function $f:\R^m\to \C$ is {\em
  real-analytic\/} on an open set $\cU\subset \R^m$ if $f$ can, in a
neighbourhood of each point in $\cU$, be represented as a convergent power
series. As will become clear soon, the ultimate proof of Theorem \ref{lemx} relies heavily on
the following refinement of Lemma \ref{lema5}.

\begin{lem}\label{lema6}
For every $p\in \Z$ and $\beta\in \R$, the function
$i_{p,\beta}$ is real-analytic on $(-1,1)$.
\end{lem}

\begin{proof}
As $i_{p,0}$ is constant, and thus trivially real-analytic, henceforth
assume $\beta \ne 0$. 
By Lemma \ref{lema5}, the function $f:\T \to \C$ with $f(t)=
i_{p,\beta}\bigl(\cos (\pi t) \bigr)$ is well-defined and continuous.
Hence it can be represented, at least in the
$L^2(\lambda_{\T})$-sense, as a Fourier series $f(t) \sim \sum_{k\in \Z} c_k e^{2\pi \imath k t}$ where, for every
$k\in \Z$,
\begin{align*}
c_k & = \int_{\T } f(t) e^{-2\pi \imath k t} \, {\rm d} t
 = \int_{\T^2} e ^{-2\pi \imath k t_1 + 4 \pi \imath |p|t_2 + 2\imath \beta \ln | \cos (\pi
  t_1) + \cos (2\pi t_2) |} \, {\rm d} t \\[1mm]
& = \int_{\T^2} e^{4\pi \imath |p| (t_1 - t_2) - 4\pi \imath k (t_1 + t_2) + 2\imath
  \beta \ln |2\cos (2\pi t_1) \cos (2\pi t_2)|} \, {\rm
  d}t \\[1mm]
& = e^{\imath \beta \ln 4} \int_{\T } e^{4\pi \imath (|p|-k)t + 2 \imath \beta
  \ln |\cos (2\pi t)|} \, {\rm d}t  \int_{\T }
e^{4\pi \imath (|p|+k) t + 2\imath \beta \ln|\cos(2 \pi t)|} \, {\rm
  d} t  \\[1mm]
& = e^{\imath \beta \ln 4} I_{|p|-k, \beta} I_{|p|+k, \beta} \, .
\end{align*}
Since $c_{-k} = c_k$, the Fourier series of $f$ is 
$$
c_0 + 2  \sum\nolimits_{n\in \N} c_n \cos (2\pi n
t ) = c_0 + 2 \sum\nolimits_{n=1}^{\infty} c_n T_{2n}\bigl( \cos (\pi t )
\bigr) \, , 
$$ 
and since furthermore
$$
|c_n| = |I_{n-|p|, \beta} I_{n+|p|, \beta}| =\frac{\beta
  \tanh (\pi \beta)}{\pi \sqrt{(n^2 + p^2 + \beta^2)^2 - 4 n^2 p^2}} = \cO
(n^{-2}) \, , \quad \mbox {as } n \to \infty \, ,
$$
and hence $\sum_{n=1}^{\infty}|c_n|<+\infty$, this series converges
uniformly on $\T$, by the Weierstrass M-test.
It follows that $i_{p,\beta} (x) = c_0 + 2 \sum\nolimits_{n=1}^{\infty} c_n
T_{2n}(x)$ uniformly in $x\in [-1,1]$.

For every $y \in (-1,1)$, consider now the auxiliary function
$$
h(x,y):= 2 \sum\nolimits_{n = 1+ |p|}^{\infty} c_n T_{2n}(x)
y^n \, .
$$
Note that $
i_{p,\beta} (x) = c_0 + 2 \sum\nolimits_{n=1}^{|p|} c_n
T_{2n}(x) + \lim\nolimits_{y \uparrow 1} h(x,y) $ uniformly in $x\in [-1,1]$.
In addition, introduce an analytic function on the open unit disc as
\begin{equation}\label{ea6}
H(z) := \sum\nolimits_{n=1+ |p|}^{\infty} c_n z^n \, , \quad \forall
z\in \C :  |z| < 1 \, ,
\end{equation}
and observe that 
\begin{align*}
H&( z)  = z^{1+|p|} \sum\nolimits_{n=0}^{\infty} c_{n+1+ |p|} z^n =
e^{\imath \beta \ln 4} z^{1+ |p|} \sum\nolimits_{n=0}^{\infty} I_{n+1,\beta} I_{n+1+
2|p|, \beta} z^n \\[2mm]
& = e^{-\imath \beta \ln 4}z^{1+|p|} \frac{\bigl( 2 \imath \beta \Gamma (2\imath
  \beta)\bigr)^2}{\bigl( \imath \beta \Gamma (\imath \beta) \bigr)^4}
\sum\nolimits_{n=0}^{\infty}
\frac{(-\imath \beta)_{n+1}(-\imath \beta)_{n+1+2|p|}}{(1+\imath \beta)_{n+1} (1+\imath \beta)_{n+1+
    2|p|}} z^n \\[2mm]
& = e^{-\imath \beta \ln 4}z^{1+ |p|} \frac{\bigl( 2 \imath \beta \Gamma (2\imath
  \beta)\bigr)^2}{\bigl( \imath \beta \Gamma (\imath \beta) \bigr)^4}
\cdot \frac{(\imath \beta)^2}{(1+\imath \beta)^2}\cdot \frac{(1-\imath
  \beta)_{2|p|}}{(2+\imath \beta)_{2|p|}}
\cdot \\
& \qquad \qquad 
\cdot
\sum\nolimits_{n=0}^{\infty}
\frac{(1-\imath \beta)_n(1+2|p|-\imath \beta)_{n}}{(2+\imath \beta)_n (2+2|p|+\imath \beta)_{n}}
z^n \\[2mm]
& =
\frac{4 e^{-\imath \beta \ln 4} \Gamma (2\imath \beta)^2 (1-\imath
  \beta)_{2|p|}}{(1+\imath \beta)^2 \Gamma(\imath \beta)^4 (2+\imath
  \beta)_{2|p|}} \cdot \\
& \qquad \qquad \cdot
\prescript{}{3}{F}^{ }_{2} (1-\imath \beta ,1+ 2|p| - \imath \beta , 1; 2+\imath \beta, 2+2|p|+ \imath \beta ; z) z^{1+|p|}\, ;
\end{align*}
here the standard notation for (generalized) hypergeometric functions
has been used, see e.g.\ \cite[Ch.II]{MOS} or \cite[Ch.16]{NIST}. Recall that $\prescript{}{3}{F}^{
}_{2}$ is an analytic function on $\C \setminus [1,+\infty)$, that is, on the entire complex plane minus a
cut from $1$ to $\infty$ along the positive real axis. Hence $H$ as
given by (\ref{ea6}) can be extended analytically to $\C \setminus [1,+\infty)$ as well. Observe now that 
\begin{align*}
H (e^{2\pi \imath t} y )  +   H(e^{-2\pi \imath t} y) & = 2 \! \sum\nolimits_{n=1+ |p|}^{\infty} \! \! c_n T_{2n} \bigl(\cos (\pi t )
\bigr) y^n \\
& = h\bigl( \cos (\pi t) , y \bigr) \, , \quad 
\forall t \in \T , y \in (-1,1) \, .
\end{align*}
It follows that, for all $x\in [-1,1]$,
\begin{align*}
i_{p,\beta}(x) & = c_0 + 2 \sum\nolimits_{n=1}^{|p|} c_n
T_{2n}(x) \: + \\
& \quad \!\!
\: + \:  \lim \nolimits_{y \uparrow 1} \left\{
H \bigl( (2x^2 - 1 + 2\imath x \sqrt{1-x^2}) y \bigr) +
H \bigl( (2x^2 - 1  - 2\imath x \sqrt{1-x^2}) y \bigr)
\right\} \\
& = c_0 + 2 \sum\nolimits_{n=1}^{|p|} c_n T_{2n}(x) \:
+ \\
&  \quad \!\! + 
H \bigl( 2x^2 - 1 + 2\imath x \sqrt{1-x^2}\, \bigr) +
H \bigl( 2x^2 - 1 - 2\imath x \sqrt{1-x^2}\, \bigr)  \, .
\end{align*}
Note now that $2z^2 - 1\pm 2\imath z \sqrt{1-z^2}\not \in [1,+\infty)$
whenever $|z|<1$. The function
$$
z\mapsto c_0 + 2 \sum\nolimits_{n=1}^{|p|} c_n T_{2n}(z)
+ 
H \bigl( 2z^2 - 1 + 2\imath z \sqrt{1-z^2}\, \bigr) +
H \bigl( 2z^2 - 1 - 2\imath z \sqrt{1-z^2}\, \bigr)  \, ,
$$
therefore, is analytic on the open unit disc and coincides with
$i_{p,\beta}$ on $\{z: |z|<1\}\cap \R = (-1,1)$. Thus $i_{p,\beta}$ is real-analytic on
$(-1,1)$, and in fact $i_{p,\beta}(x) = \sum_{n=0}^{\infty}
i_{p,\beta}^{(n)}(0) x^n/n!$ for all $x\in (-1,1)$.
\end{proof}

\begin{rem}
Since $t \mapsto x+\cos (2\pi t)$ does not change sign on $\T$ whenever $|x|>1$, it
is clear from (\ref{funca}) that the function $i_{p,\beta}$ is real-analytic on $\R
\setminus [-1,1]$ as well.
\end{rem}

For every $d\in \N$, define a non-empty open subset of $\R^d$ as
$$
\cE_d:= \left\{ u \in \R^d : \exists j \in \{1, \ldots , d\}\:
  \mbox{\rm with } |u_{j}|> \sum\nolimits_{k\ne j} |u_k|
\right\} \, .
$$
Geometrically, $\cE_d$ is the disjoint union of $2d$ open cones. For
example, $\cE_1 = \R \setminus \{0\}$ and $\cE_2 = \{u\in \R^2 :
|u_1| \ne  |u_2|\}$, hence $\cE_d$ is also dense in $\R^d$ for
$d=1,2$. For $d\ge 3$ this is no longer the case. In fact, a simple
calculation shows that 
$$
\frac{\mbox{Leb} (\cE_d \cap [-1,1]^d)}{\mbox{Leb} ([-1,1]^d)} =
\frac{2^d/\Gamma (d)}{2^d} = \frac1{\Gamma(d)} \,  , \quad \forall d \in \N
\, ,
$$
and so the (relative) portion of $\R^d$ taken up by $\cE_d$ decays
rapidly with growing $d$.

In order to utilize Lemma \ref{lema6} for a proof of Theorem
\ref{lemx}, given any $p_1, \ldots, p_d \in \Z$ and
$\beta \in \R$, recall the map $\Lambda_u$ from
(\ref{ea2}) and consider the integral
\begin{align}\label{ea7}
J= J(u)  & = \widehat{\lambda_{\T^d} \circ \Lambda_{u}^{-1}} (2) =
\int_{\T} e^{4\pi \imath t} \, {\rm d} (\lambda_{\T^d} \circ \Lambda_u^{-1})
(t) \nonumber \\
& = \int_{\T^d} e^{4\pi \imath (p_1 t_1 + \ldots p_d t_d + \beta \ln |u_1
  \cos (2\pi t_1) + \ldots + u_d \cos (2\pi t_d)|)} \, {\rm d}t \, . 
\end{align}
An important consequence of Lemma \ref{lema6} is

\begin{lem}\label{lema7}
For every $p_1, \ldots , p_d \in \Z$ and $\beta \in
\R\setminus \{0\}$, the function $u\mapsto J(u)$ given by
{\rm (\ref{ea7})} is
real-analytic and non-constant on each connected component of $\cE_d$. 
\end{lem}

\begin{proof}
If $d=1$ then, as seen in essence already in the proof of Lemma \ref{lema4},
$$
u_1 \mapsto J(u_1) = \int_{\T } e^{4\pi \imath p_1 t + 4 \pi \imath \beta \ln |u_1
  \cos (2\pi t)|} \, {\rm d}t = e^{4\pi i \beta \ln
  |u_1|} I_{p_1 , 2 \pi \beta}
$$
is real-analytic and non-constant on each of the two connected
parts of $\R \setminus \{0\} =
\cE_1$.

Assume in turn that $d\ge 2$. As the roles of $t_1, \ldots , t_d$ can
be interchanged in (\ref{ea7}), assume w.l.o.g.\ that $u_d \ne 0$. Since
$J(\pm u_1, \ldots , \pm u_d)= J(u_1, \ldots , u_d)$ for
all $u \in \R^d$ and every possible combination of $+$ and $-$
signs, and since also
$$
J(u) = e^{4\pi i \beta \ln |u_d|} J \left( 
  \frac{u_1}{u_d} , \ldots , \frac{u_{d-1}}{u_d}, 1\right) \, ,
$$
it suffices to show that $\widetilde{J} = \widetilde{J}(u):= J(u_1,
\ldots , u_{d-1}, 1)$ is real-analytic and non-constant on
$\widetilde{\cE}_{d-1}:= \{u \in \R^{d-1}:
\sum_{j=1}^{d-1}|u_j|<1\}$. To this end note first that
$$
\widetilde{J}(u) = \int_{\T^{d-1}} e^{4\pi \imath (p_1 t_1 + \ldots +
  p_{d-1}t_{d-1})}
i_{p_d, 2 \pi \beta} \bigl( 
u_1 \cos (2\pi t_1) + \ldots + u_{d-1} \cos (2\pi t_{d-1})
\bigr) \, {\rm d}t \, .
$$
With Lemma \ref{lema5} and the Dominated Convergence Theorem, it is
clear that $\widetilde{J}$ is continuous on $\R^{d-1}$. Recall from
the proof of Lemma \ref{lema6} that $i_{p, \beta}$ can be represented
by a power series, namely $i_{p,\beta}(x) = \sum_{n=0}^{\infty}
i_{p, \beta}^{(n)}(0) x^n/n!$ for all $p\in \Z$, $\beta \in \R$ and $|x|<1$. For every $u\in
\widetilde{\cE}_{d-1}$, therefore,
\begin{align}\label{eqamons}
\widetilde{J}(u) &  = 
\int_{\T^{d-1}} e^{4\pi \imath (p_1 t_1 + \ldots + p_{d-1}t_{d-1})}
\sum\nolimits_{n=0}^{\infty} \frac{i_{p_d, 2\pi \beta}^{(n)}(0)}{n!}
\left( 
\sum\nolimits_{j=1}^{d-1} u_{j} \cos (2\pi t_{j})
\right)^n \nonumber
\\[2mm]
& = \sum\nolimits_{n=0}^{\infty}  \frac{ i_{p_d, 2\pi \beta}^{(2n)} (0) }{2^{2n}}  \sum\nolimits_{|\nu| = n} \left\{
  \prod\nolimits_{j=1}^{d-1} \frac{u_j^{2\nu_j}}{(2\nu_j)!} \left(  \!\! \begin{array}{c} 2\nu_j \\ \nu_j \! + \! |p_j|  \end{array}
  \!\! \right) \right\} \, ,
\end{align}
where the standard notation for multi-indices $\nu = (\nu_1, \ldots ,
\nu_{d-1})\in (\N_0)^{d-1}$ has been used, see e.g.\
\cite[pp.25--29]{KP}. Thus $\widetilde{J}$ is real-analytic on
$\widetilde{\cE}_{d-1}$, by \cite[Prop.2.2.7]{KP}.

It remains to show that $\widetilde{J}$ is non-constant on
$\widetilde{\cE}_{d-1}$. Consider first the case of $d=2$, for which
(\ref{eqamons}) takes the form
\begin{equation}\label{ea9}
\widetilde{J}(u_1) 
= \sum\nolimits_{n=|p_1|}^{\infty}  \frac{ i_{p_2, 2\pi \beta}^{(2n)} (0) }{2^{2n}} \left(  \!\! \begin{array}{c} 2n \\  n \! + \! |p_1|  \end{array}
  \!\! \right) \frac{u_1^{2n}}{(2n)!} \, , \quad \forall u_1 \in
\widetilde{\cE}_1 = (-1,1) \, .
\end{equation}
Recall that $u_1 \mapsto \widetilde{J}(u_1)$ is continuous. If $p_1
\ne 0$ then $\widetilde{J}(0)=0$ whereas
\begin{align*}
\widetilde{J}(1) & = \int_{\T^2} e^{4\pi \imath (p_1 t_1 + p_2 t_2 + \beta \ln |
  \cos (2\pi t_1) + \cos (2\pi t_2)|)} \, {\rm d} t  \\[1mm]
& = \int_{\T^2} e^{4\pi \imath (p_1 (t_1 - t_2) + p_2 (t_1 + t_2) + \beta \ln |2
  \cos (2\pi t_1)  \cos (2\pi t_2)|)} \, {\rm d} t \\
& = e^{4\pi \imath
  \beta \ln 2} I_{p_1 + p_2 , 2\pi \beta} I_{p_1 - p_2 , 2 \pi\beta} \ne 0 \, ,
\end{align*}
since $\beta \ne 0$.  If, on the other hand, $p_1 = 0$ then $\widetilde{J}(0) = I_{p_2, 2\pi \beta}$,
while $\widetilde{J}(1) = e^{4\pi \imath \beta \ln 2} I_{p_2, 2\pi \beta}^2 \ne \widetilde{J}(0)$.
In either case, therefore, $u_1 \mapsto \widetilde{J}(u_1)$ is
non-constant on $\widetilde{\cE}_1=(-1,1)$. This concludes the proof for $d=2$.

Finally, to deal with the case of $d\ge 3$, note first that the above
argument for $d=2$ really shows
that, given any $p\in \Z$ and $\beta \in \R \setminus \{0\}$, the number $i_{p, 2\pi
  \beta}^{(2n)}(0)$ is non-zero for infinitely many $n\in
\N_0$. (Otherwise, by (\ref{ea9}), the function $u_1 \mapsto
\widetilde{J}(u_1)$ would be constant for $|p_1|$ sufficiently large, which
has just been shown not to be the case.) But then
$$
\widetilde{J}(u)  = \sum\nolimits_{n=|p_1| + \ldots +
  |p_{d-1}|}^{\infty} \frac{  i_{p_d, 2\pi \beta}^{(2n)} (0) }{2^{2n}} \sum\nolimits_{|\nu| = n} \left\{
  \prod\nolimits_{j=1}^{d-1} \frac{u_j^{2\nu_j}}{(2\nu_j)!} \left(  \!\! \begin{array}{c} 2\nu_j \\ \nu_j \! + \! |p_j|  \end{array}
  \!\! \right) \right\} 
$$
is obviously non-constant on $\widetilde{\cE}_{d-1}$.
\end{proof}

Given $p_1, \ldots , p_d \in \Z$ and $\beta \in \R$, denote by $\cD_d$
the set of all $u\in \R^d$ for which
$\lambda_{\T^d} \circ \Lambda_u^{-1}$ coincides with $\lambda_{\T}$,
i.e., let $
\cD_d =  \{u \in \R^d : \lambda_{\T^d} \circ \Lambda_u^{-1} =
\lambda_{\T}  \}$.
An immediate consequence of Lemma \ref{lema7} is

\begin{lem}\label{lema8}
For every $p_1, \ldots , p_d \in \Z$ and $\beta \in \R
\setminus \{0\}$ the set $\cD_d \cap \cE_d\subset \R^d$ is nowhere dense and has
Lebesgue measure zero.
\end{lem}

\begin{proof}
This is clear from the fact that $\cD_d \cap \cE_d \subset  \{u \in
\cE_d: J(u) = 0\}$. As $u\mapsto J(u)$ is
real-analytic and non-constant on each component of $\cE_d$, the zero-locus of $J$
on $\cE_d$ is nowhere dense and has Lebesgue measure zero; see e.g.\
\cite[Lem.19]{BHKR} or \cite[Sec.4.1]{KP}.
\end{proof}

At long last, the {\em Proof of Theorem \ref{lemx}\/} has become very simple: Since $\cD_d \cap \cE_d$ is nowhere dense, $\cE_d \setminus
\cD_d \ne \varnothing$, and $\lambda_{\T^d}\circ \Lambda_u^{-1} \ne
\lambda_{\T}$ for every $u\in \cE_d \setminus \cD_d$, by the 
definition of $\cD_d$. \hfill $\qed$

\begin{rem}\label{rema9}
(i) Since $\cE_1$ and $\cE_2$ are dense in $\R$ and $\R^2$,
respectively, the set $\cD_d$ is nowhere dense in $\R^d$ for
$d=1,2$ whenever $\beta \ne 0$. It may be conjectured that $\cD_d$ is
nowhere dense (and has Lebesgue measure zero) for $d\ge 3$ also; no
proof of, or counter-example to this conjecture is known to the authors.

(ii) Note that $\lambda_{\T^d}
\circ \Lambda_u^{-1} = \lambda_{\T }$ if, for some $j\in \{1,\ldots , d\}$, both $p_j \ne 0$
and $\beta u_j =0$. Thus
\begin{equation}\label{ea10}
\bigcup\nolimits_{j:p_j \ne 0} \left\{  u\in \R^d : \beta u_j = 0 \right\}
\subset \cD_d \, ,
\end{equation}
and hence for $\beta \ne 0$ the set $\cD_d$ contains the union of at most $d$
coordinate hyper-planes. Beyond the conjecture formulated in (i), it is tempting to speculate
whether in fact equality holds in (\ref{ea10}) always, i.e.\ for any
$p_1, \ldots , p_d\in \Z$ and $\beta \in \R$ ---
as it does for $\beta =0$ (trivial) and $d=1$ (Lemma \ref{lema4}).
Obviously, equality in (\ref{ea10}) would establish a much stronger
version of Theorem \ref{lemx}.

(iii) Even if the set $\cD_d \subset \R^d$ is indeed nowhere dense and has
Lebesgue measure zero for every $d\in \N$, as conjectured in (i), for
large values of $d$
the equality $\lambda_{\T^d}\circ \Lambda_u^{-1} = \lambda_{\T }$,
though generically false, is nevertheless often true approximately --- in some
sense, and quite independently of the specific values of $p_1, \ldots
, p_d \in \Z$ and $\beta \in
\R \setminus \{0\}$. Under mild conditions on these parameters, this
observation can easily be made rigorous as follows:
Assume, for instance, that the integer sequence $(p_n)$ is not
identically zero, say $p_1 \ne 0$ for convenience, and $\beta \ne
0$. Also assume that
\begin{equation}\label{ea13}
(u_n) \: \, \mbox{\rm  is a bounded sequence in} \: \, \R \,  \: \mbox{\rm with } \sum\nolimits_{n=1}^{\infty} u_n^2 = +\infty \, .
\end{equation}
If
$u_{1} =0$ then $\lambda_{\T^d} \circ \Lambda_u^{-1} =
\lambda_{\T }$ for all $d\in \N$. On the other hand, if $u_{1}
\ne 0$, let $\sigma_d := \sqrt{1 +
  \sum_{j=2}^d u_j^2}$ and observe that, for every $k\in
\Z \setminus \{0\}$,
\begin{align*}
 \widehat{\lambda_{\T^d}  \circ \Lambda_u^{-1}} (k)  & = e^{2\pi \imath k
   \beta \ln \sigma_d}  \cdot \\
& \quad \cdot  \int_{\T^d} 
e^{2\pi \imath k \left(p_{1} t_{1} + \sum_{j=2}^d p_j t_j + \beta
  \ln \left|u_1/ \sigma_d  \cos (2\pi t_1) + \sum_{j=2}^d u_j / \sigma_d
  \cos (2\pi t_j) \right|\right)} \, {\rm d} t \, .
\end{align*}
Since $\sigma_d \to +\infty$ as $d  \to  \infty$ yet $(u_n)$ is
bounded, it follows from the Central
Limit Theorem (see e.g.\ \cite[Sec.9.1]{CT}) that $\lim_{d\to \infty}
\widehat{\lambda_{\T^d} \circ \Lambda_u^{-1}} (k)  =  0$. Under the mild assumption
(\ref{ea13}), therefore, $
\lim\nolimits_{d \to \infty} \lambda_{\T^d} \circ \Lambda_u^{-1} =
\lambda_{\T } $ in $\cP (\T )$
in the sense of weak convergence of
probability measures. Informally put, the probability measure
$\lambda_{\T^d}\circ \Lambda_u^{-1}$ typically differs but little from
$\lambda_{\T}$ whenever $d$ is large.

(iv) The above {\em proof\/} of Theorem \ref{lemx} relies heavily on
specific properties of the logarithm, notably on the fact that $\ln
|xy| = \ln |x| + \ln |y|$ whenever $xy\ne 0$. It seems plausible,
however, that the {\em conclusion\/} of that theorem may remain valid if
the function $\ln |\cdot|$ in (\ref{ea2}) is replaced by virtually
any non-constant function that is real-analytic on $\R \setminus
\{0\}$ and has $0$ as a mild singularity. Establishing such a much
more general version of Theorem \ref{lemx} will likely require a
conceptual approach quite different from the rather computational
strategy pursued herein.
\end{rem}

\end{appendix}

\end{document}